\documentclass{article}
\usepackage{amsmath, amsthm, amssymb, mathrsfs, enumerate, color, graphicx, caption, subcaption}

\usepackage{framed}

\DeclareMathAlphabet{\mathonebb}{U}{bbold}{m}{n}
\newcommand{\one}{\ensuremath{\mathonebb{1}}}

\def\vp{\varphi}

\def\R{{\mathbb R}}

\def\RR{{\mathbb R}}

\def\e{\varepsilon}
\def\vp{\varphi}

\def\oh{\frac{1}{2}}

\def\la{\langle}
\def\ra{\rangle}
\def\({\left(}
\def\){\right)}
\newcommand{\abs}[1]{\left\vert{#1}\right\vert}
\newcommand{\norm}[1]{\left\Vert{#1}\right\Vert}
\newcommand{\He}{\mathbf{H}_{ext}}
\newcommand{\GM}{\mathcal{G}_{\mathcal{M},\kappa}}
\newcommand{\mM}{\mathcal{M}}
\newcommand{\mH}{\mathcal{H}}

\newcommand{\bH}{\mathbf{H}}
\newcommand{\bA}{\mathbf{A}}
\newcommand{\be}{\begin{equation}}
\newcommand{\ee}{\end{equation}}
\DeclareMathOperator{\esssup}{ess\, sup}
\DeclareMathOperator{\essinf}{ess\, inf}

\newtheorem{lem}{Lemma}[section]

\newtheorem{theorem}{Theorem}[section]

\newtheorem{remark}{Remark}[section]

\newtheorem{lemma}[theorem]{Lemma}

\newtheorem{prop}[theorem]{Proposition}

\title{Persistence of superconductivity in thin shells beyond $H_{c1}$}

\author{
Andres Contreras 
\thanks{Department of Mathematics, University of Toronto,
The Fields Institute for Research in Mathematical Sciences. {\ttfamily ancontre@umail.iu.edu}}
\thanks{Department of Mathematical Sciences, New Mexico State University.}
\and
Xavier Lamy
\thanks{Universit\'e de Lyon, Institut Camille Jordan,  CNRS UMR 5208, Universit\'e Lyon 1,  France. {\ttfamily xlamy@math.univ-lyon1.fr}
}
\thanks{Department of Mathematics and Statistics, McMaster University.} 
}

\begin{document}

\maketitle

\begin{abstract}
In Ginzburg-Landau theory, a strong magnetic field is responsible for
the breakdown of superconductivity. This work is concerned with the
identification of the region where superconductivity persists, in a
thin shell superconductor modeled by a compact surface $\mathcal M\subset\mathbb R^3$, as the intensity $h$ of the external magnetic field is raised above $H_{c1}$. Using a mean field reduction approach devised by Sandier and Serfaty as the Ginzburg-Landau parameter $\kappa$ goes to infinity,
 we are led to studying a \textit{two-sided} obstacle problem.
 We show that superconductivity survives in a neighborhood of size
$\(H_{c1}/h\)^{1/3}$ of the zero locus of the normal component
$H$ of the field. We also describe intermediate regimes, focusing first on a 
symmetric model problem. In the general case, we prove that  a striking phenomenon we
call {\it freezing of the boundary} takes place: one component of the
superconductivity region is insensitive to small changes in the field.
\end{abstract}

\section{Introduction}\label{s_intro}

Let $\mathcal{M}$ be a compact surface homeomorphic to $\mathbb{S}^2$, embedded in
$\mathbb{R}^3.$ For $\kappa, h>0$ and $\mathbf{A}$ a vector field on
$\mathcal{M},$ we consider the Ginzburg-Landau functional
$\mathcal{G}_{\mathcal{M},\kappa}:H^1(\mathcal{M};\mathbb{C})\to\mathbb{R}_{+},$
\be\label{GLS}
\GM(\psi)=\int_\mM\(\abs{\nabla_\mM-ih\mathbf{A}\psi}^2+\frac{\kappa^2}{2}\(\abs{\psi}^2-1\)^2\)d\mH_{\mM}^2(x).
\ee
The functional $\GM$ arises as the $\Gamma$-limit (see \cite{CS}) of the full 3d  Ginzburg-Landau energy

\begin{equation}\label{GL3D}
\begin{split}
G_{\e,\kappa}(\psi,A)& =\frac{1}{\varepsilon}\Bigg[\int_{\Omega_\e}\(\abs{(\nabla-iA)\psi}^2+\frac{\kappa^2}{2}\(\abs{\psi}^2-1\)^2\)dx \\
&\quad +\int_{\R^3}\abs{\nabla\times A-\He}^2 dx\Bigg].
\end{split}
\end{equation}
where for all $\e>0$ sufficiently small, $\Omega_\e$ corresponds to a uniform tubular neighborhood of $\mathcal{M}.$
In \eqref{GL3D} $\He$ is the external magnetic field. As $\e \to 0 ,$ the field completely penetrates the sample which then implies that in the $\Gamma$-limit $A$ is prescribed to be equal to $\mathbf{A},$ the tangential component of a divergence free vector field $\bA^e$ such that $\nabla\times h \bA^e=\He.$

A central question in Ginzburg-Landau theory is the determination of the so-called {\it critical fields}.
The first critical field corresponds to the appearance of zeros of $\psi$ carrying non-trivial degree -- called vortices in this context -- in minimizers of the energy.

The analysis in \cite{CS} includes the computation of the first critical field of a thin shell of a surface of revolution subject to a constant vertical field which turns out to be surprisingly simple and depending only on an intrinsic quantity, in the $ \kappa\to\infty$ limit:
$$
H_{c_1}\sim\(\frac{4\pi}{\mbox{Area of }\mM}\)\ln \kappa.
$$ 
This result is extended in \cite{C}, to general surfaces and magnetic fields. For a fixed field $\bH^e$, an external magnetic field of the form $\He=h( \kappa)\bH^e=h( \kappa)\nabla\times \bA^e$ is considered. Then the first critical field is
$$
H_{c_1}\sim\frac{1}{\max_\mM *F-\min_\mM *F}\ln  \kappa,
$$
where $d^{*}F=*d*F=\bA$ and $*$ denotes the Hodge star-operator.
In fact, the study shows also that, somewhat remarkably, not all fields $\bH^e$ give rise to a first critical field. This phenomenon is related to the geometry and relative location of $\mM$ with respect to $\bH^e.$ For $\bH^e$ that yield a finite $H_{c_1},$ the topological obstruction imposed by $\mM$ implying that the total degree of $\frac{\psi}{\abs{\psi}}$ is zero is used in \cite {C} to show that there is an even number of vortices in minimizers of $\GM,$ half with positive degree, half with negative degree concentrating respectively on the set where $*F$ achieves its minimum and maximum.  The optimal number $2n$ and location of vortices and anti-vortices in $\mM$ is established in \cite{C} for values of $h(k)$ slightly above $H_{c_1}$ and in addition it is shown that if the minimum and maximum of $*F$ is attained at finitely many points then the two sets of vortices  minimize, independently, a renormalized energy.

The results in \cite{C} and \cite{CS} cover only a moderate regime; in these works the intensity of the applied  field  is $H_{c_1} +\mathcal{O}(\ln\ln \kappa )$ and thus the number of vortices remains bounded as $\kappa$ goes to infinity.

Once the value of $h$ becomes much larger than $H_{c_1},$ that is there is a constant $C>0$ such that $h-H_{c_1}\geq C\ln\kappa,$ then the number of vortices in minimizers diverges as $\kappa\to\infty$. For even larger $h$,  superconductivity persists  only in a narrow region in the sample.

In the case of an infinite cylinder whose cross section is a domain $\Omega\subseteq\R^2$ and for constant applied fields parallel to the axis of the cylinder a reduction to a two-dimensional problem is possible.
In this case it is known that as the intensity increases superconductivity is lost  in the bulk and only a thin superconductivity region near $\partial\Omega$ persists (see Chapter 7 in \cite{SS2}).
For much higher values still, superconductivity is completely lost: this value is known as $H_{c_3}$ and is estimated by a delicate spectral analysis of the magnetic Laplacian operator as in the monograph \cite{fournaishelffer10}.

In our setting, corresponding to the above functional $\GM$ \eqref{GLS} on the compact surface $\mathcal M$, there is no boundary, so what happens to the superconductivity region is not obvious. Another  crucial difference lies in the behaviour of the (normalized) magnetic field $H$ induced on $\mathcal M$, which is the normal component of $\mathbf H^e$, or equivalently $H\, d\mathcal H^2_{\mathcal M}=d\mathbf A$ (viewing $\mathbf A$ as a 1-form). Namely, in our case, $H$ vanishes and changes sign. The spectral analysis in \cite{montgomery95} therefore suggests that superconductivity should persist near the set $\lbrace H=0\rbrace$, where the external magnetic field is tangent to the surface $\mathcal M$. In \cite{pankwek02} the authors study the case of a vanishing magnetic field in the infinite cylinder model, and observe indeed nucleation of superconductivity near the zero locus of the magnetic field, for very high values of the applied field (near the putative $H_{c_3}$) under the condition that the gradient of the magnetic field does not vanish on its zero locus. The problem of the determination of the upper critical field for vanishing fields remains largely open otherwise.  Here, we are concerned with much lower values of the applied field: a main motivation of this work is to understand the transition from the vortexless to normal state regimes.

Another interesting difference is the fact that in the infinite cylinder model only positive  vortices exist  and so the location and growth of the vortex region is always ruled by the competing effects of  mutual repulsion, and confinement provided by the external field. In the present setting, this is no longer the case. Vortices of positive and negative degree must coexist and so repulsion and attraction are common features  of the relative placement of vortices in $\mathcal{M}$, this without taking into account the external field.

 In this way,  the shrinking of the superconductivity region is a multifaceted phenomenon. Moreover, the problems mentioned in the characterization of this region are present even in the most emblematic case of a constant external field $\mathbf H^e$: the region of persistence of superconductivity does not only depend on the field and on the topology of $\mM,$ but also on extrinsic geometric properties of the surface; the relative position of $\mathcal{M}$ with respect to $\mathbf H^e$ affects $H$ and therefore the zero locus of the induced field.

In the present work we address the question of identifying the region where superconductivity persists in the $\kappa\to\infty$ limit, when 
$$\frac{H_{c1}}{h}$$ is small; we show that as this quantity gets small superconductivity persists in a small neighborhood of the place where the applied field is tangential to the sample, provided the field satisfies a generic non-degeneracy condition(see \eqref{Hnondegen} below). 
Another thrust of this work is aimed at uncovering some new intermediate regimes only present in this setting, when the normal component of the external field changes sign multiples times. In the model problem of a surface of revolution and constant vertical field, we  identify several structural transitions undergone by the superconductivity region. Furthermore, we observe a new phenomenon which we refer to as {\it freezing of the boundary}, where a component of the vortex region stops growing even after increasing the intensity of the external field. This phenomenon holds in great generality (not only in the surface of revolution case), as is shown at the end of section \ref{s_interm}.

To carry out our analysis  we start by using a reduction to a mean field model, first derived rigorously in \cite{SS1}. More precisely, if we write a critical point $\psi$ of $\GM$ in polar form $\psi=\rho e^{i\phi},$ variations of the phase yield $d(\rho^2(d\phi-hd^{*}F))=0,$ and because $H_{dr}^1(\mM)=0$ this implies there is a $V$ such that $*dV=\rho^2(d\phi-hd^{*}F).$
Taking $V=hW$, the function $W$ is expected to minimize
\begin{equation}\label{mf}
\int_\mM\abs{\nabla_\mM W}^2 d\mH_\mM^2
+\frac{\ln  \kappa}{h}\int_\mM\abs{-\Delta_\mM W+\Delta_\mM *F}d\mH_\mM^2.
\end{equation}

The details of this mean field reduction can be found in \cite{SS1} in the case of a positive external field applied in a bounded planar domain. However, the analysis in \cite{SS1} does not handle the additional restriction of total zero mass which affects the construction of an upper bound in this setting.  The steps needed to extend the proof to the present case are included in Appendix~\ref{a_meanfield}.

The measure $-\Delta_{\mathcal{M}} V +\Delta_{\mathcal{M}} *F$ can be interpreted as the normalized measure generated by the vortices. 
On the other hand,
we observe that  
\begin{equation*}
\Delta_{\mathcal{M}} *F \, d\mathcal H^2_{\mathcal M} = d*d*F =d\mathbf A= Hd\mathcal H^2_{\mathcal M},
\end{equation*}
where the function $H$ is the normal component of the external magnetic field $\mathbf H^e$ relative to $\mathcal{M}$. In what follows we refer to $H$ simply as \textit{the} magnetic field, and we assume that $H\in C^1(\mathcal M)$. Moreover, we drop the explicit dependence on $\mathcal{M}$ in expressions like $\Delta_{\mathcal M}$, $\nabla_{\mathcal M}$.

Before we  state our main result we make the following assumption: there exists $\beta>0$ such that
\begin{equation} \label{defbeta}\lim_{\kappa\to \infty }\frac{\ln \kappa}{h}=\beta.\end{equation}
 Once the connection to the mean field problem \eqref{mf} is established we proceed to locate very precisely the region of persistence of superconductivity, that is, the region $SC_\beta$ where the vorticity measure $-\Delta V + H$ vanishes. We find that this region corresponds to a $\beta^{\frac{1}{3}}$ neighborhood of the set where $H$ vanishes, in the $\beta\to 0$ limit. More precisely,

\begin{theorem}\label{simplifiedmain} Under the nondegeneracy assumption that $\nabla H$ is nowhere vanishing on $\{H=0\}, $ there exists $C>0$ independent of $\beta$ such that the superconductivity region $SC_\beta$ is contained in $\{x\in \mathcal{M}: d(x, \{H=0\})<C\beta^{\frac{1}{3}}\}$, and contains  $\{x\in \mathcal{M}: d(x, \{H=0\})<C^{-1}\beta^{\frac{1}{3}}\}$,
for $\beta$ sufficiently small.
\end{theorem}

The nondegeneracy assumption on $H$ implies that the set $\lbrace H=0\rbrace$ is a finite union of smooth closed curves. It is the same assumption as the one made in \cite{montgomery95,pankwek02} for the study of the third critical field $H_{c_3}$.

To prove Theorem \ref{simplifiedmain} we reformulate the mean field approximation as an obstacle problem, and construct comparison functions. We note that a construction in the same spirit was carried out in \cite[Appendix~A]{VorLatt} for the planar Ginzburg-Landau model in a different context. In our case however the construction is not immediate, because our obstacle problem is two-sided and our magnetic field $H$ changes sign. Indeed, our proof makes use of a comparison principle for two sided obstacle problems proved in \cite{dalmasomoscovivaldi} which allows to compare solutions to obstacle problems corresponding to different data $H$. Hence the comparison functions will not be merely ``super- or sub-solutions'' of our problem, but actual solutions of modified problems. In particular they have to be quite regular. As a consequence, we cannot use functions of the  distance to $\lbrace H=0\rbrace$ as comparison functions. We have to use a particular coordinate system near each component of $\{H=0\}$ and explicitly build local functions satisfying local obstacle problems with appropriate modifications of $H$. Pasting these constructions we are able to  appeal to \cite{dalmasomoscovivaldi} to obtain the desired estimates. In so doing we note a key feature of the proof, related to the fact that the obstacle problem is two-sided: the barriers thus obtained cannot be used independently to get neither the inner nor the outer bound separately, but together they yield the conclusion of the theorem. This is explained in more detail in section \ref{s_smallbeta}.

Thanks to Theorem~\ref{simplifiedmain}, we have  a clear picture of the superconductivity region for $\beta\to 0$: it is a union of tubular neigborhoods of the connected components of $\lbrace H=0\rbrace$. In particular, the superconductivity region has at least as many connected components as $\lbrace H=0\rbrace$. On the other hand, we also have a clear picture of the superconductivity region as $\beta\to\beta_c$, where positive (resp. negative) vortices are concentrated near the points where $*F$ achieves its maximum (resp. minimum). In particular, the superconductivity region has, generically, one connected component. In the last part of this work, we investigate the intermediate regimes. If $\lbrace H=0\rbrace$ has more than one connected component, transitions \textit{have} to occur: when $\beta$ crosses some critical value, the number of connected components of $SC_\beta$ changes. 

Studying such transitions, and determining the values of $\beta$ at which they occur, seems out of our reach in all generality. That is why we concentrate first on a simple model problem. We consider a surface of revolution around the vertical axis $\mathbf{e_z}$, and assume that the external magnetic field $\mathbf H^e = \mathbf{e_z}$ is vertical and constant. (In fact in Section~\ref{ss_symcase}, more general magnetic fields are considered.) In that case, the induced field $H$ on $\mathcal M$ is just $H=\mathbf{e_z}\cdot \nu$, where $\nu$ is an outward normal vector on $\mathcal M$. The set $\lbrace H=0\rbrace$ consists exactly of the points where $\mathbf{e_z}$ is tangent to $\mathcal M$, and it is a union of circles. Note that $H$ has to change sign an odd number of times, since $H=-1$ at the `south pole' and $+1$ at the `north pole', thus there are an odd number of those circles. As explained above, interesting transitions happen when $\lbrace H=0\rbrace$ has more than one connected component. Therefore we focus on the simplest non-trivial situation, which corresponds to $\lbrace H=0\rbrace$ consisting of three circles.
We state loosely here the result that we obtain for that simple model problem in Section~\ref{ss_symcase} (see Figure~\ref{figreg}).

\begin{prop}\label{loose1d}
There exists $\beta_c>\beta_1^* \geq \beta_2^* >0$ such that
\begin{itemize}
\item for $\beta\in (\beta^*_1,\beta_c)$, $SC_\beta$ has one connected component,
\item for $\beta\in (\beta^*_2,\beta^*_1)$, $SC_\beta$ has two connected components,
\item for $\beta\in (0,\beta^*_2)$, $SC_\beta$ has three connected components.
\end{itemize}
Moreover, for $\beta\in(\beta^*_2,\beta^*_1)$, one connected component of $SC_\beta$ remains constant.
\end{prop}

\begin{figure}[ht]
\begin{center}
\begin{subfigure}{.32\textwidth}
\includegraphics[width=\textwidth]{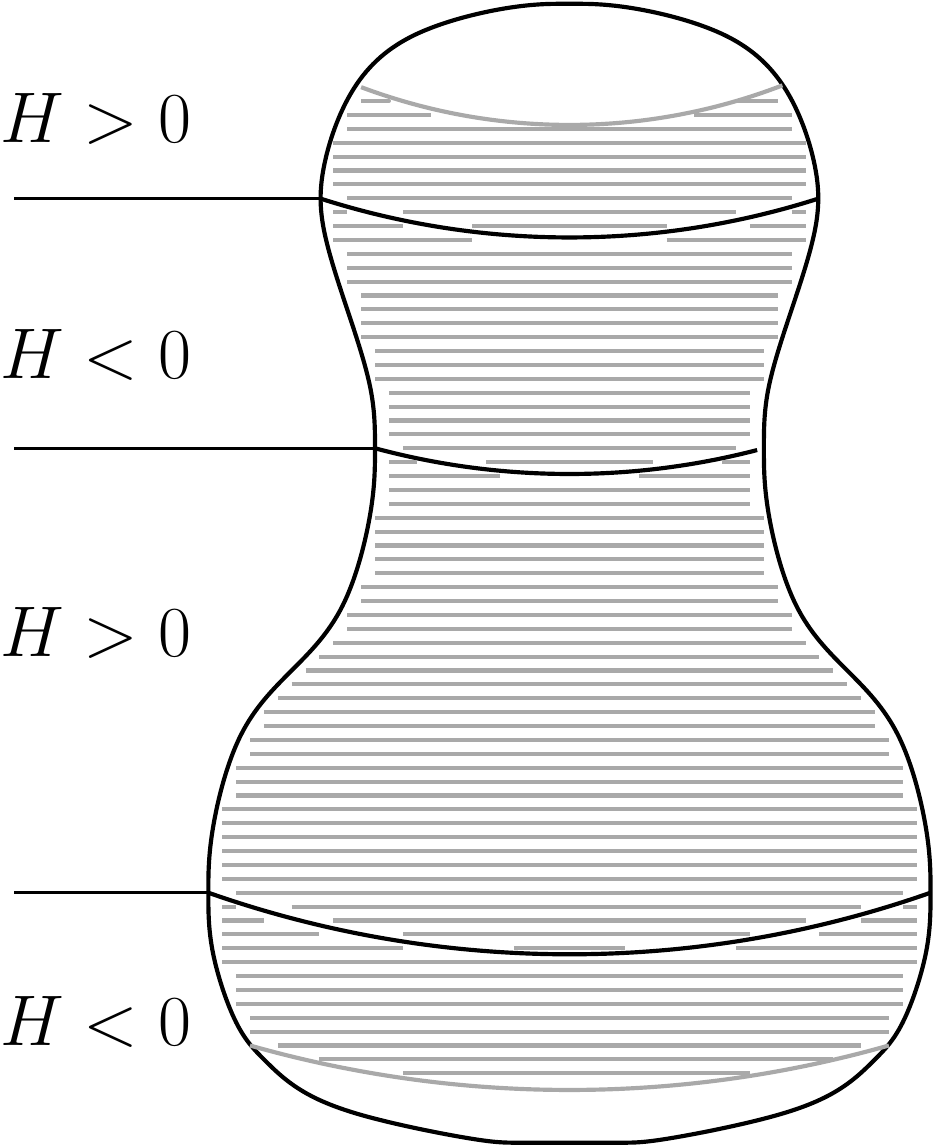}
\caption{$\beta\in (\beta_1^*,\beta_c)$}
\label{figreg1}
\end{subfigure}
\hspace{.05\textwidth}
\begin{subfigure}{.24\textwidth}
\includegraphics[width=\textwidth]{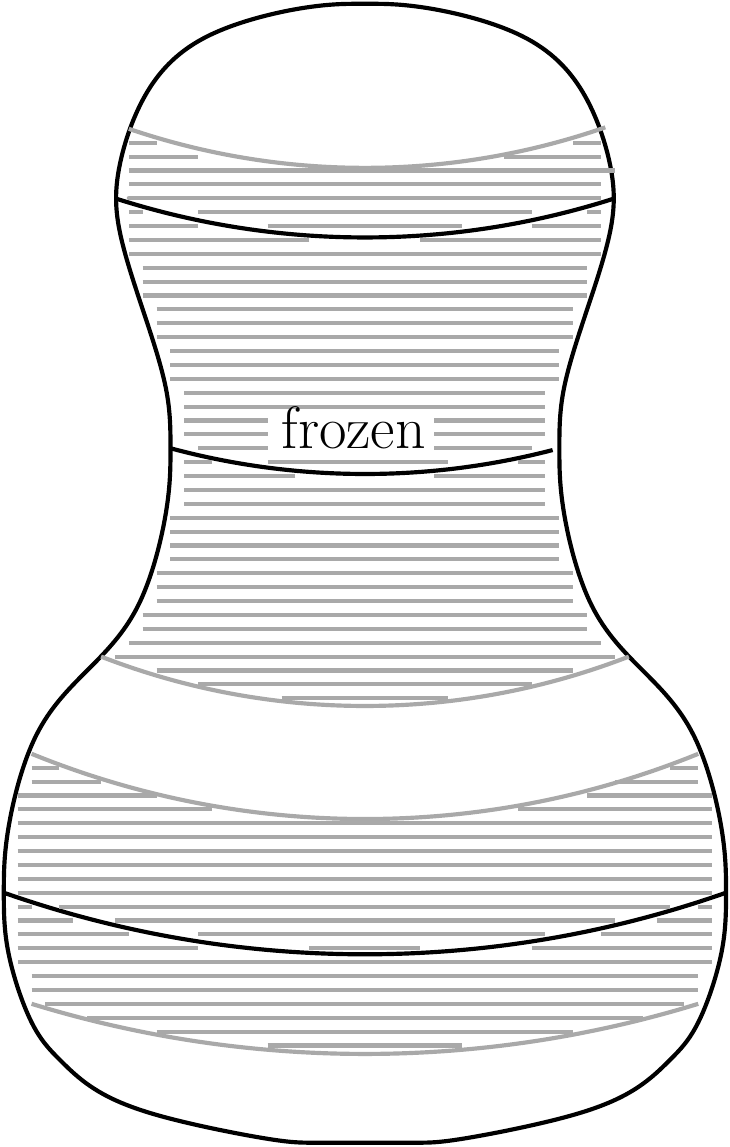}
\caption{$\beta\in (\beta_2^*,\beta_1^*)$}
\label{figreg2}
\end{subfigure}
\hspace{.05\textwidth}
\begin{subfigure}{.24\textwidth}
\includegraphics[width=\textwidth]{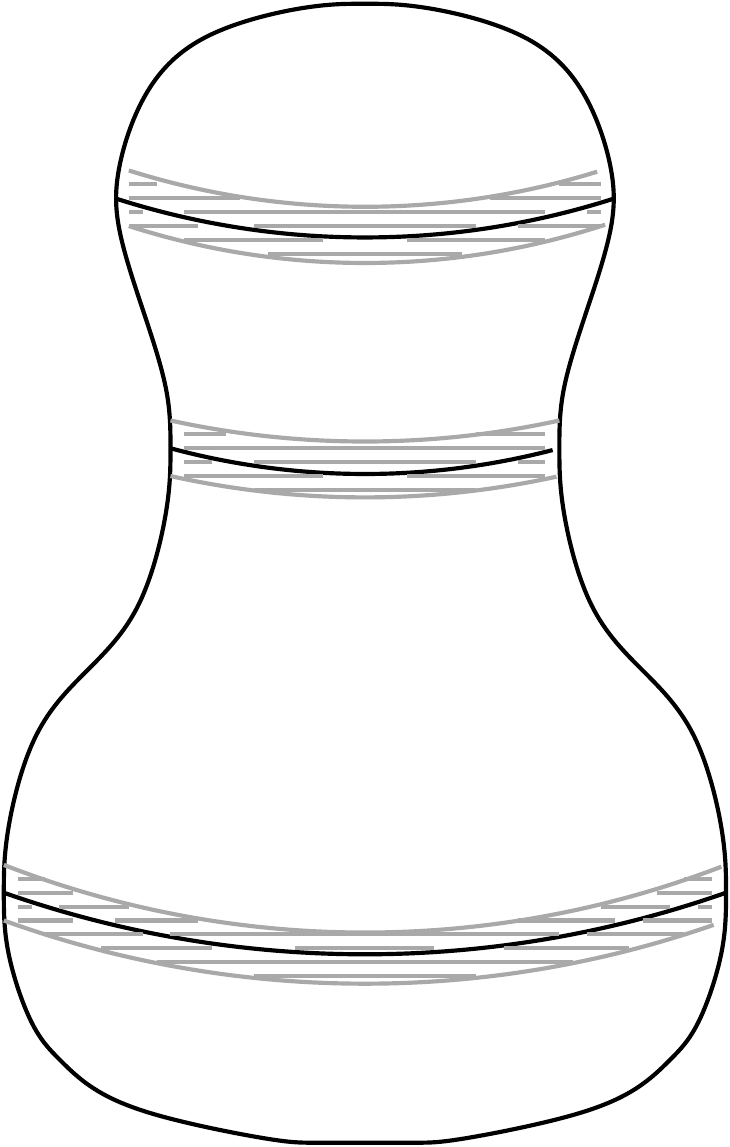}
\caption{$\beta\in (0,\beta_2^*)$}
\label{figreg3}
\end{subfigure}
\caption{The region $SC_\beta$ in the three regimes of Proposition~\ref{loose1d}}
\label{figreg}
\end{center}
\end{figure}

The most striking part of Proposition~\ref{loose1d} is the appearance of an intermediate regime in which one connected component of $SC_\beta$ remains constant: one part of the free boundary is \textit{frozen}. In Section~\ref{ss_freez} we identify the features responsible for such `freezing' of the boundary and prove a similar `freezing property' in a general (non-symmetric) setting (see Proposition~\ref{prop_freez}).

An other interesting outcome of the precise version of Proposition~\ref{loose1d} (Proposition~\ref{prop_regimes1d} in Section~\ref{ss_symcase}) are the expressions of the critical values $\beta^*_1$ and $\beta^*_2$, in terms of integral quantities involving $\mathbf A$ and the parametrization of $\mathcal M$. Transfering these conditions to a general non-symmetric setting seems far from obvious and constitutes an interesting challenge.

The plan of the paper is as follows.
In the next section we collect some basic properties of solutions to an obstacle problem that serves as the starting point in our analysis. In section \ref{s_smallbeta} we identify the thin region of superconductivity when $\beta$ is small. In section \ref{s_interm} we turn to the symmetric situation and identify in Proposition~\ref{prop_regimes1d} the further transitions as $\beta$ decreases to zero from $\beta_c=\max(*F)-\min(*F).$ We also prove the `freezing of the boundary' property at the end of  section \ref{s_interm}. 

\subsubsection*{Acknowledgements}

A.C. is
supported by the Fields postdoctoral fellowship. He wishes to thank
Robert L. Jerrard for his incredible support. X.L. is supported by a Labex Milyon doctoral mobility scholarship for his stay at McMaster University. He wishes to thank his Ph.D. advisor Petru Mironescu and his McMaster supervisors Stanley Alama and Lia Bronsard for their great support.

\section{The obstacle problem}\label{s_obstacle}
This preamble is devoted to the derivation of the obstacle problem dual to the mean field approximation. We also prove some basic results we will need later on. We think it is worthwhile recording these properties because in our setting, even in the by now classical application of the duality theorem which allows for the obstacle problem formulation, there is an inherent degeneracy we have to account for which is not present in other similar results in the literature. 

In the first part of this section we show that -- as in \cite[Chapter~7]{SS2} -- the minimizer of 
\begin{equation}\label{Ebeta}
E_\beta (V)=\int_{\mathcal M} \abs{\nabla V}^2  +\beta 
 \int_{\mathcal M} \abs{ -\Delta V + H} 
\end{equation}
 is the solution of an obstacle problem, and then we study general properties of the contact set. There are two main differences with the obstacle problem arising in \cite[Chapter~7]{SS2}.
 \begin{itemize}
 \item In our case there are no boundary conditions and the minimizer is well-defined only up to a constant. We need to deal with this degeneracy.
 \item While in \cite[Chapter~7]{SS2} the obstacle problem is one-sided, we have to consider a two-sided obstacle problem. This is due to the fact that, in our case, the magnetic field $H$ changes sign.
 \end{itemize}
 
The functional $E_\beta$ is, under assumption \eqref{defbeta}, the limit of the sequence of energies considered in \eqref{mf}. The link between $E_\beta$ and the superconductivity region is, as mentioned in the introduction, proved in appendix A.

\subsection{Derivation of the obstacle problem}\label{ss_derivobstacle}

\begin{prop}\label{prop_obstacle}
Let $\beta>0$. A function $V_0\in H^1(\mathcal M)$ minimizes $E_\beta$ \eqref{Ebeta} if and only if $V_0$ minimizes
\begin{equation}\label{F(V)}
\mathcal F (V) = \int_{\mathcal M}\!\left(\abs{\nabla V}^2 + 2  HV\right)
\end{equation}
among all $V\in H^1(\mathcal M)$ such that $(\esssup V-\essinf V)\leq \beta$.
\end{prop}

\begin{remark}\label{rem_obstacle}
Since the functional $\mathcal F(V)$ is translation invariant, $V_0$ coincides, up to a constant, with any minimizer of the two-sided obstacle problem
\begin{equation*}
\min \left\lbrace \int_{\mathcal M}\!\left(\abs{\nabla V}^2 + 2 HV\right) \colon V\in H^1(\mathcal M),\: \abs{V}\leq\beta /2 \right\rbrace.
\end{equation*}
Moreover, recalling that $H=\Delta *F$, this obstacle problem can also be rephrased as
\begin{equation}\label{obstacle*F}
\min \left\lbrace \int_{\mathcal M}\!\abs{\nabla(V-*F)}^2 \colon V\in H^1(\mathcal M),\: \abs{V}
 \leq \beta/2 \right\rbrace.
\end{equation}
The fact that minimizers coincide only up to a constant does not matter, since the physically relevant object is the vorticity measure $-\Delta V +H$. Moreover, it is easy to check that, if the obstacle problem \eqref{obstacle*F} admits a solution $V$ that `touches' the obstacles, i.e. satisfies $\max V -\min V=\beta$, then this solution is unique because any other solution differs from it by a constant, which has to be zero. On the other hand, a solution satisfying $\max V-\min V <\beta$ would have to be $V=*F+\alpha$ for some constant $\alpha$. Therefore, for $\beta\leq \max *F-\min *F$ the solution is unique.
\end{remark}

The proof of Proposition~\ref{prop_obstacle} relies on the following classical result of convex analysis (easily deduced from \cite{rockafellar66} or \cite[Theorem~1.12]{B}). 

\begin{lemma}\label{lem_dual}
Let $\mathcal H$ be a Hilbert space and $\varphi: \mathcal H\to\RR\cup\{+\infty\}$ be a convex lower semi-continuous function.
Then the minimizers of the problems 
\begin{equation*}\min_{x\in \mathcal H}\(\oh\norm{x}^2_{\mathcal H}+\vp(x)\)\quad \mbox{ and } \quad
\min_{y\in \mathcal H}\(\oh\norm{y}^2_{\mathcal H}+\vp^{*}(-y)\)
\end{equation*}
coincide, where $\vp^{*}$ denotes the Fenchel conjugate of $\varphi$, 
\begin{equation*}
\vp^{*}(y):=\sup_{z\in \mathcal H}\la y,z\ra_{\mathcal H}-\vp(z).
\end{equation*}
\end{lemma}

\begin{proof}[Proof of Proposition~\ref{prop_obstacle}:]
We apply Lemma~\ref{lem_dual} in the Hilbert space
\begin{equation*}
\mathcal H:=\dot H^1(\mathcal M)=\left\lbrace V\in H^1(\mathcal M)\colon \int_{\mathcal M} V =0\right\rbrace,
\end{equation*}
endowed with the norm $\norm{V}^2=\int \abs{\nabla V}^2$, to the function
\begin{equation}\label{phibeta}
\varphi(V)=\varphi_\beta(V)=\frac\beta 2 \int_{\mathcal M} \abs{-\Delta V +H }.
\end{equation}
In formula \eqref{phibeta}, it is implicit that $\varphi(V)=+\infty$ if $\mu = -\Delta V + H$ is not a Radon measure. Note that, when $\mu$ \textit{is} a Radon measure, it must have zero average $\int\mu=0$, since $\mu=\Delta (*F-V)$.

We compute the Fenchel conjugate of $\varphi$. It holds
\begin{align*}
\varphi^*(V) & 
= \sup_{U\in \mathcal H } \left\lbrace \int_{\mathcal M}\! \nabla V \cdot \nabla U -\frac\beta 2 \int_{\mathcal M}\! \abs{-\Delta U +H} \right\rbrace \\
& = -\int_{\mathcal M} HV 
 + \sup_{U\in \mathcal H} \left\lbrace \int_{\mathcal M}\! (-\Delta U + H)V -\frac\beta 2 \int_{\mathcal M}\! \abs{-\Delta U +H} \right\rbrace \\
& = -\int_{\mathcal M} H V + \sup_{ \int P =0} \left\lbrace \int_{\mathcal M}\left( P V-\frac\beta 2 \abs{P}\right) \right\rbrace.
\end{align*}
In the last equality, the supremum may -- by a density argument -- be taken over all $L^2$ functions $P$ with zero average.

If $(\esssup V-\essinf V)\leq\beta$, then $\abs{V+\alpha}\leq \beta/2$ for some $\alpha\in\R$, so that
\begin{equation*}
\int_{\mathcal M}\!\left( P V-\frac\beta 2 \abs{P}\right) = \int_{\mathcal M}\! \left( (V+\alpha)P -\frac\beta 2 |P| \right) \leq 0,
\end{equation*}
and in that case
\begin{equation*}
\varphi^*(V)=-\int_{\mathcal M}\! HV.
\end{equation*}
On the other hand, if $(\esssup V-\essinf V)>\beta$, then up to translating $V$ we may assume that $\lbrace V>\beta/2\rbrace$ and $\lbrace V<-\beta/2\rbrace$ have positive measures. It is then easy to construct a function $P$ supported in those sets, such that $\int P=0$, $\int |P|=1$, and $\int PV>\beta/2$. Using $\lambda P$ as a test function for arbitrary $\lambda>0$, we deduce that $\varphi^*(V)=+\infty$.

From Lemma~\ref{lem_dual} it follows that $V_0\in \dot H^1(\mathcal M)$ minimizes $E_\beta$ if and only if $V_0$ minimizes 
\begin{equation*}
\frac 12 \int _{\mathcal M}\! \abs{\nabla V}^2 + \int_{\mathcal M}\! HV 
\end{equation*}
among $V\in \dot H^1(\mathcal M)$ such that $\esssup V-\essinf V\leq \beta$. Since both problems are invariant under addition of a constant, the restriction to the space $\dot H^1(\mathcal M)$ can be relaxed to obtain Proposition~\ref{prop_obstacle}.
\end{proof}

\subsection{Basic properties}\label{ss_basicprop}

In this section we concentrate on the obstacle problem
\begin{equation}\label{obstacleH}
\min \left\lbrace \int_{\mathcal M}\!\left( \abs{\nabla V}^2 + 2HV\right)\colon V\in H^1(\mathcal M),\: |V|\leq \beta/2 \right\rbrace.
\end{equation}
We recall the classical interpretation of \eqref{obstacleH} as a free boundary problem, and establish a monotonicity property of the free boundary.

The first step to these basic properties is the reformulation of the obstacle problem \eqref{obstacleH} as a variational inequality: a function $V\in H^1(\mathcal M)$ solves \eqref{obstacleH} if and only if $\abs{V}\leq \beta/2$ and
\begin{equation}\label{varineq}
\int_{\mathcal M}\! \nabla V \cdot \nabla (W-V) \geq -\int_{\mathcal M}\! H(W-V) \qquad\forall W\in H^1(\mathcal M),\: |W|\leq \beta/2.
\end{equation}
The proof of this weak formulation is elementary and can be found in many textbooks on convex analysis. See for instance \cite{rodrigues}.

Next we recall the standard reformulation of \eqref{varineq} as a free boundary problem.
\begin{lemma}\label{lem_freebound}
A function $V\in H^1(\mathcal M)$ with $\abs{V}\leq \beta/2$ solves \eqref{obstacleH} or equivalently \eqref{varineq} if and only if 
\begin{equation}\label{freebound}
\left\lbrace
\begin{aligned}
 V&\in W^{2,p}(\mathcal M),\quad 1<p<\infty,\\
 \Delta V & = H \quad \text{in }\lbrace \abs{V}<\beta/2\rbrace,\\
 0 & \geq H \quad \text{in }\lbrace V=\beta/2 \rbrace,\\
 0 & \leq H \quad \text{in }\lbrace V=-\beta/2 \rbrace.
\end{aligned}
\right.
\end{equation}
In particular $V\in C^{1,\alpha}(\mathcal M)$, so that at every regular point of the free boundaries $\partial \lbrace V=\pm\beta/2\rbrace$, the function $V$ satisfies the overdetermining boundary conditions $V=\pm\beta/2$ and $\partial V / \partial\nu =0$.
\end{lemma}

The only non-elementary part of Lemma~\ref{lem_freebound}  is the $W^{2,p}$ regularity of the solution. There are many textbooks on the subject, at least for the one-sided obstacle problem. See for instance \cite{petrosyan}. For the two-sided obstacle problem we refer to \cite{dalmasomoscovivaldi}, in which the authors study very fine properties of two-sided obstacle problems. For the convenience of the reader, since most of the literature  only deals with a one-sided obstacle (and \cite{dalmasomoscovivaldi} reaches far beyond this quite elementary result), we provide a sketch of the proof in Appendix~\ref{a_prooffreebound}.

Recall that in our case, $\mu =-\Delta V + H$ represents the vorticity measure. In light of Lemma~\ref{lem_freebound}, this measure is supported in $\lbrace V=\pm\beta/2\rbrace$. In that region, vortices are distributed with density $H$. 

For $\beta > \beta_c$, where
\begin{equation}\label{beta_c}
\beta_c := \max (*F) - \min (*F),
\end{equation}
the function $*F + \alpha$ solves the obstacle problem \eqref{obstacleH}, as long as the constant $\alpha$ satisfies $\max (*F) -\beta/2 \leq \alpha \leq \min (*F) +\beta/2$, and the vorticity measure $-\Delta V+H$ is identically zero.

For $\beta\leq \beta_c$, the solution $V=V_\beta$ of the obstacle problem~\eqref{obstacleH} must satisfy
\begin{equation*}
\max V_\beta -\min V_\beta =\beta,
\end{equation*}
and therefore is unique (see Remark~\ref{rem_obstacle}). Recall that the superconductivity region $SC_\beta$ is defined as the set where the vorticity measure $-\Delta V +H$ vanishes. According to Lemma~\ref{lem_freebound}, that region is exactly
\begin{equation}\label{SCbeta}
SC_\beta =\lbrace \abs{V_\beta}<\beta/2 \rbrace.
\end{equation}

A first basic property of the superconductivity region $SC_\beta$ is its monotonicity.
\begin{prop}\label{prop_monot}
For any $0<\beta_1 < \beta_2\leq \beta_c$, it holds
\begin{equation*}
SC_{\beta_1} \subset SC_{\beta_2}.
\end{equation*}
\end{prop}
In other words, increasing the intensity of the applied magnetic field shrinks the region of persisting superconductivity, which consistant with physical intuition.
Since we have to deal with a two-sided obstacle problem, this monotonicity property is not as obvious as in \cite[Chapter~7]{SS2}. To prove it, we use a comparison principle for two-sided obstacle problems \cite[Lemma~2.1]{dalmasomoscovivaldi}. We state and prove here a particular form that will also be useful later on.

\begin{lemma}\label{lem_comparison}
Let $H_1\geq H_2$ be bounded, real-valued functions on $\mathcal M$. Let also $\alpha_1\leq\alpha_2$ and $\beta_1\leq\beta_2$ be real numbers. For $j=1,2$, let $V_j\in H^1(\mathcal M)$ solve respectively the obstacle problems
\begin{equation*}
\min \left\lbrace \int_{\mathcal M}\! \left(\abs{\nabla V}^2 + 2H_jV\right)\colon \alpha_j\leq V\leq \beta_j\right\rbrace.
\end{equation*}
Then either $V_1-V_2$ is constant, or $V_1\leq V_2$.
\end{lemma}
\begin{proof}
For the convenience of the reader, we provide here the elementary proof, which consists in remarking that
\begin{equation*}
W_1 = \min(V_1,V_2)\quad\text{and }
W_2 = \max(V_1,V_2)
\end{equation*}
are admissible test functions in the variational inequalities 
\begin{equation*}
 \int_{\mathcal M}\! \nabla V_j \cdot \nabla (W_j-V_j)\geq -\int_{\mathcal M}\! H_j(W_j-V_j),\qquad\forall W_j\in H^1,\: \alpha_j\leq W_j\leq \beta_j.
\end{equation*}
 Substracting the resulting inequalities, we obtain
\begin{equation*}
\int_{\mathcal M}\abs{\nabla (V_1-V_2)_+}^2 \leq \int_{\mathcal M}(H_2-H_1)(V_1-V_2)_+ \leq 0,
\end{equation*}
where $(V_1-V_2)_+=\max(V_1-V_2,0)$. We conclude that $(V_1-V_2)_+$ is a constant function.
\end{proof}

With Lemma~\ref{lem_comparison} at hand, we may prove the monotonicity of the superconductivity region.

\begin{proof}[Proof of Proposition~\ref{prop_monot}:]
Let $V_1$ and $V_2$ denote the solution of the obstacle problem \eqref{obstacleH} corresponding respectively to $\beta=\beta_1$ and $\beta=\beta_2$. Let 
\begin{equation*}
\widetilde V_1 =V_1 +\beta_1/2,\quad\text{and}\quad \widetilde V_2=V_2+\beta_2/2,
\end{equation*}
so that for $j=1,2$, $\widetilde V_j$ solves the obstacle problem
\begin{equation*}
\min\left\lbrace \int_{\mathcal M}\!\left(\abs{\nabla V}^2 + 2HV\right)\colon 0\leq V\leq \beta_j\right\rbrace.
\end{equation*}
Therefore, applying Lemma~\ref{lem_comparison} with $H_1=H_2=H$, $\alpha_1=\alpha_2=0$ and $\beta_1\leq \beta_2$, we deduce that
\begin{equation*}
V_1+\beta_1/2\leq V_2 +\beta_2/2.
\end{equation*}
(If $\widetilde V_1- \widetilde V_2$ is constant, then $\beta_2 = \max V_1-\min V_1 = \beta_1$.)
In particular, we obtain that
\begin{equation*}
\lbrace V_1 >-\beta_1/2 \rbrace \subset \lbrace V_2 > -\beta_2/2\rbrace.
\end{equation*}
By a similar argument, we show that
\begin{equation*}
\lbrace V_1 <\beta_1/2 \rbrace \subset \lbrace V_2 < \beta_2/2\rbrace,
\end{equation*}
and conclude that $SC_{\beta_1}\subset SC_{\beta_2}$.
\end{proof}

\begin{remark}\label{rem_cont}
It follows from the above proof  that
\begin{equation*}
|V_1-V_2|\leq (\beta_2-\beta_1)/2,
\end{equation*}
thus proving the continuity of $\beta\mapsto V_\beta$ for $0\leq \beta\leq\beta_c$.
\end{remark}

\section{The small $\beta$ limit}\label{s_smallbeta}
In this section we study what happens to the superconductivity set when the intensity of the field is high enough to confine it in a narrow region. 
We make the (generic) non degeneracy assumption that
\begin{equation}\label{Hnondegen}
|H|+|\nabla H| > 0\quad\text{in }\mathcal M.
\end{equation}
In other words, $\nabla H\neq 0$ in $\lbrace H=0\rbrace$. This implies in particular that the set $\Sigma:=\lbrace H=0\rbrace$ where the magnetic field vanishes is a finite disjoint union of smooth closed curves. We also note that condition \eqref{Hnondegen} also implies that we are not in the situation where not even the first critical field is defined(see \cite{C}, Theorem 3.1).

Let us say a few words here about the nondegeneracy assumption \eqref{Hnondegen}. This is the  same nondegeneracy assumption that has been considered in works on the spectral analysis of the magnetic Laplacian \cite{montgomery95} and on higher applied magnetic fields in Ginzburg-Landau \cite{pankwek02,attar14}.
Moreover, we emphasize that \eqref{Hnondegen} is a generic assumption, in the following sense.
\begin{lem}\label{generic}
The set of $H$ satisfying \eqref{Hnondegen} is open and dense in $C^1(\mathcal M)$.
\end{lem}
\begin{proof}
The fact that \eqref{Hnondegen} is an open condition in $C^1(\mathcal M)$ is clear. The density follows from a transversality theorem by Quinn \cite[Theorem~3]{quinn70}, applied to the $C^1$ map
\begin{equation*}
\Phi\colon C^1(\mathcal M) \times \mathcal M\to\mathbb R,\quad (H,x)\mapsto H(x).
\end{equation*}
For a function $H\in C^1(\mathcal M)$, \eqref{Hnondegen} is equivalent to $\Phi(H,\cdot)$ being transverse to $\lbrace 0\rbrace$. Clearly, $D_H\Phi(H,x)=I_{C^1(\mathcal M)}$ is Fredholm, and $\Phi$ is transverse to $\lbrace 0\rbrace$. Therefore, the set of $H$ such that $\Phi(H,\cdot)$ is transverse to $\lbrace 0\rbrace$ is dense in $C^1(\mathcal M)$. 
\end{proof}

We are interested in the behavior, as $\beta\to 0$, of the superconductivity region $SC_\beta$ \eqref{SCbeta}.

We let $d\colon\mathcal M\to\mathbb R_+$ denote the distance function to the set $\Sigma=\lbrace H=0\rbrace$, that is
\begin{equation}\label{defd}
d(x)=\mathrm{dist}(x,\lbrace H=0\rbrace).
\end{equation}
In this context we characterize the behavior of $SC_\beta$ in terms of the function $d$, as follows(this is a more explicit version of Theorem \ref{simplifiedmain}).

\begin{theorem}\label{thm_scbeta}
Under the non-degeneracy assumption \eqref{Hnondegen} on the magnetic field, there exists $\beta_0>0$ and $C>0$ such that, for $\beta\in(0,\beta_0)$, \begin{equation}\label{boundsscbeta}
\left\lbrace d\leq \frac 1C \beta^{1/3}\right\rbrace \subset SC_\beta \subset \left\lbrace d\leq C\beta^{1/3} \right\rbrace,
\end{equation}
where $SC_\beta$ is the superconductivity region  \eqref{SCbeta}, and $d$ denotes the distance to the zero locus of the magnetic field \eqref{defd}.
\end{theorem}

In the proof we construct explicit solutions to modified obstacle problems, in order to apply the comparison principle Lemma~\ref{lem_comparison}. The comparison functions are constructed locally near each component $\Gamma$ of $\lbrace H=0\rbrace$, and then we need to extend and paste these functions and the associated modified obstacle problem data. Although the construction looks local, it is worth noting that we really need to make it near \textit{every} component $\Gamma$ of $\lbrace H=0\rbrace$. Otherwise the pasting would not provide us with obstacle problems comparable to the original one, because a solution has to change sign near \textit{every} curve $\Gamma$.

\begin{remark}\label{rem_outer+thick}
Another natural approach to proving Theorem~\ref{thm_scbeta} would be to construct separate comparison functions in $\lbrace H>0\rbrace$ and $\lbrace H<0\rbrace$. In those regions, the obstacle problem becomes one-sided, so that more standard constructions with a classical comparison principle can be made. On the other hand, there is no boundary conditions in those regions, so that such a construction would only provide us with the outer bound 
\begin{equation}\label{outer}
SC_\beta \subset \lbrace d\leq C \beta^{1/3}\rbrace.
\end{equation}
 To obtain the bounds \eqref{boundsscbeta} which show that
the superconductivity set extends to \underline{both} sides of the zero locus of $H$ by a $\beta^{\frac{1}{3}}$ margin, it seems that we really have to appeal to the comparison principle for two-sided obstacle problems. 
However, if we would just content ourselves with showing that the superconductivity set had `thickness' proportional to $\beta^{\frac{1}{3}}$, namely
\begin{equation}\label{thickness}
\mathrm{dist}(\lbrace V=\beta/2\rbrace,\lbrace V=-\beta/2\rbrace)\geq c\beta^{1/3},
\end{equation}
there would be a simpler way. In fact \eqref{thickness} can be directly inferred from \eqref{outer}. This is a simple consequence of the interpolated elliptic estimate (see \cite[Appendix~A]{BBH1})
\begin{equation}\label{interpolestim}
\norm{\nabla V}^2_\infty \leq C \norm{\Delta V}_\infty \norm{V}_\infty,
\end{equation}
which implies, since $|V|\leq\beta$ and 
$|\Delta V|=|H\one_{SC_\beta}| \leq C \beta^{1/3},$
that
\begin{equation}\label{gradientbound}
|\nabla V|\leq C \beta^{2/3} \quad\text{in }\mathcal M.
\end{equation}
Hence, for any $x_\pm\in \lbrace V=\pm\beta/2\rbrace$ and any arc-length parametrized curve $\gamma(s)$, ($0\leq s\leq \ell$) going from $x_-$ to $x_+$, it holds
\begin{equation*}
\beta =V(x_+)-V(x_-) = \int_0^\ell \nabla V(\gamma(s))\cdot\gamma'(s) \, ds \leq  C \beta^{2/3}\ell,
\end{equation*}
so that the length of $\gamma$ satisfies $\ell\geq c\beta^{1/3}$, which proves \eqref{thickness}.
\end{remark}

Next we turn to the proof of Theorem~\ref{thm_scbeta}.

\begin{proof}[Proof of Theorem~\ref{thm_scbeta}:]
We will construct, for small enough $\beta$, bounded functions $H_1\leq H\leq H_2$, and comparison functions $V_1$ and $V_2$ of regularity $W^{2,\infty}$, satisfying for $j=1,2$,
\begin{equation}\label{eqVj}
\begin{gathered}
\Delta V_j=H_j \one_{\abs{V_j}<\beta/2},\\
 \abs{V_j}\leq \beta/2,\quad H_j\geq 0 \text{ in }\lbrace V_j=-\beta/2\rbrace,\quad H_j\leq 0\text{ in }\lbrace V_j=\beta/2\rbrace,
 \end{gathered}
\end{equation}
and the bounds
\begin{equation}\label{boundsVj}
\left\lbrace d\leq \frac 1C \beta^{1/3}\right\rbrace \subset \left\lbrace \abs{V_j}<\beta/2\right\rbrace \subset \left\lbrace d\leq C\beta^{1/3} \right\rbrace.
\end{equation}
By Lemma~\ref{lem_freebound}, \eqref{eqVj} implies that $V_j$ solves the obstacle problem \eqref{obstacleH} with $H=H_j$. Therefore we may apply the comparison principle for two-sided obstacle problems (Lemma~\ref{lem_comparison}) to conclude that $V_1\geq V \geq V_2$. In view of the bounds \eqref{boundsVj} satisfied by $V_1$ and $V_2$, this obviously implies that the superconductivity region satisfies the bounds \eqref{boundsscbeta}.

The rest of the proof is devoted to constructing $V_1$ and $V_2$. To this end we introduce good local coordinates in a neighborhood of $\Sigma=\lbrace H=0\rbrace$. Recall that, thanks to the nondegeneracy assumption \eqref{Hnondegen}, $\Sigma$ is a finite union of closed smooth curves. Let us fix one of them, $\Gamma$, together with an arc-length parametrization of it:
\begin{equation*}
\Gamma =\left\lbrace \gamma(x) \colon x\in\R/\ell\mathbb Z\right\rbrace,\quad \abs{\gamma'(x)}=1.
\end{equation*}
Let us also fix a smooth normal vector $\nu(x)$ to $\Gamma$ on $\mathcal M$, that is
\begin{equation*}
\nu(x)\in T_{\gamma(x)}\mathcal M,\quad \abs{\nu}=1,\quad \nu\cdot\gamma'=0,
\end{equation*}
and impose that $\nu(x)$ points in the direction of $\lbrace H>0\rbrace$ (since $H<0$ on one side of $\Gamma$ and $H<0$ on the other side). We introduce Fermi coordinates along $\Gamma$: for small enough $\delta$, the map
\begin{equation*}
\R/\ell\mathbb Z \times (-\delta,\delta)\to\mathcal M,\quad (x,y)\mapsto \exp_{\gamma(x)}(y\nu(x)),
\end{equation*}
is a diffeomorphism. It defines local coordinates $(x,y)$ on $\mathcal M$ in a neighborhood of $\Gamma$, in which the Laplace operator has the form
\begin{equation}\label{laplacefermi}
\Delta =\frac 1f\left(\partial_y f\partial_y +\partial_x f^{-1} \partial_x\right),
\end{equation}
where $f(x,y)=1-y\kappa(x,y)$ for some smooth function $\kappa$. Note that $y$ is nothing else than the signed distance to $\Gamma$, and in particular $|y|=d$ in a neighborhood of $\Gamma$. While this is a coordinate system that follows well the geometry of a neighborhood of $\gamma , $ we actually need one  where the Laplacian allows us to reduce our construction to a $1d$ problem. To that end let $(x,z)$ be the local coordinates where
\begin{equation}\label{zcoord}
z=y+\frac 12 y^2\kappa(x,y).
\end{equation}
Clearly the map $(x,y)\mapsto (x,z)$ is a diffeomorphism for small enough $y$, so that $(x,z)$ define indeed local coordinates on $\mathcal M$.  The reason for using the coordinates $(x,z)$ is that the Laplace operator is then approximately
\begin{equation*}
\Delta \approx \partial_x^2 + \partial_z^2,
\end{equation*}
which will allow us to obtain nice bounds for functions depending only on $z$.

Note that, since we choose the normal vector $\nu$ to point in the direction of  $\lbrace H>0\rbrace$, and since $\abs{\nabla H}\geq c>0$ in a neighborhood of $\Gamma$ thanks to the nondegeneracy assumption \eqref{Hnondegen}, it holds
\begin{equation*}
\partial_z H \geq c>0,\quad \abs{z}<\delta.
\end{equation*}
On the other hand, $\nabla H$ is bounded, so that there exist $C\geq c>0$ such that
\begin{equation}\label{Hz}
Cz\one_{z<0} + cz\one_{z>0} \leq H \leq cz\one_{z<0} + Cz\one_{z>0},\qquad \abs{z}<\delta.
\end{equation}

Next we concentrate on the construction of $V_1$ ($H_1$ will be defined accordingly). Away from the set $\Sigma$, we simply define
\begin{equation}\label{V1away}
V_1=-\mathrm{sign}(H)\beta/2 \quad\text{in }\lbrace d>\delta/2\rbrace.
\end{equation}
The interesting part is of course what happens near $\Sigma$. Near each of the smooth curves $\Gamma\subset\Sigma$, we will look for $V_1$ in the form $V_1=v(z)$, where $v$ is a $W^{2,\infty}$ function satisfying
\begin{equation}\label{V1near}
v(z)=\begin{cases}
\beta/2 & \text{ for }z<-\eta_-,\\
-\beta/2 & \text{ for }z>\eta_+,
\end{cases}
\end{equation}
for some parameters $\eta_\pm>0$ that will depend on $\beta$. A straightforward computation using \eqref{laplacefermi} and \eqref{zcoord} shows that
\begin{equation}\label{DeltaV1}
\Delta V_1 =v''(z) + z \left(g_1(x,z) v''(z) + g_2(x,z) v'(z)\right),
\end{equation}
where $g_1$ and $g_2$ are bounded functions. We are going to define in $(-\eta_-,\eta_+)$ the function $v$ so that 
\begin{equation}\label{boundsv}
v''\leq 2Cz\one_{z<0} + \frac{c}{2}z\one_{z>0},\quad |v'|= o(\beta),\; |v''|= o(\beta).
\end{equation}
We then define $H_1$ in $(-\eta_{-} , \eta_+)$ simply as $\Delta V_1 .$
Thus, recalling 
\eqref{Hz}, we will have, for small enough $\beta>0$,
\begin{equation}
\Delta V_1 = H_1\one_{|V_1|<\beta/2}\qquad\text{with }H_1\leq H\;\text{in }\lbrace -\eta_-<z<\eta_+\rbrace.
\end{equation}
It is then straightforward to extend $H_1$ to a function defined on $\mathcal M$, such that $H_1\leq H$, and having the same sign as $H$ outside of $\lbrace -\eta_-<z<\eta_+\rbrace$. The resulting $H_1$ and $V_1$ satisfy \eqref{eqVj}. 

Thus it remains to show that we can indeed define $v(z)$ in $\lbrace -\eta_-<z<\eta_+\rbrace$, satisfying the bounds \eqref{boundsv}. We look for $v$ in the form
\begin{equation}\label{formv}
v(z)=\begin{cases}
v_-(z) &\text{ for }-\eta_-<z<0,\\
v_+(z) &\text{ for }0<z<\eta_+,
\end{cases}
\quad\text{with }v_\pm(z)\text{ polynomial.}
\end{equation}
First of all, for $v$ to be of class $W^{2,\infty}$ around the points $\pm\eta_\pm$, we should impose
\begin{equation}\label{regv}
v_-(-\eta_-)=\beta/2,\quad v_+(\eta_+)=-\beta/2,\quad v_-'(-\eta_-)=v_+'(\eta_+)=0.
\end{equation}
Thus we take $v_\pm$ to be of the form
\begin{equation}\label{v+v-}
\begin{aligned}
v_-(z)&=(z + \eta_-)^2(A_-z+B_-) +\frac{\beta}{2} \\
&= A_-z^3 + (B_-+2\eta_-A_-)z^2 + (2\eta_-B_- +\eta_-^2A_-)z +\eta_-^2B_-+\frac{\beta}{2} ,\\
v_+(z)&=(z-\eta_+)^2 (A_+z +B_+) -\frac\beta 2 \\
&= A_+z^3 + (B_+-2\eta_+A_+)z^2 +(-2\eta_+B_+ +\eta_+^2 A_+)z +\eta_+^2 B_+ -\frac\beta 2.
\end{aligned}
\end{equation}
For $v$ to be of class $W^{2,\infty}$ around $z=0$, we have to impose
\begin{equation}\label{condv0}
\begin{gathered}
\eta_-^2B_-+\frac{\beta}{2} = \eta_+^2 B_+ -\frac\beta 2,  \\
2\eta_-B_- +\eta_-^2A_- =-2\eta_+B_+ +\eta_+^2 A_+. 
\end{gathered}
\end{equation}
We also need to ensure that
\begin{equation}\label{v''}
v''\leq 2Cz\one_{z<0} + \frac{c}{2}z\one_{z>0},
\end{equation}
so we impose
\begin{equation}\label{condv''}
6A_- = 2C,\quad 6A_+ = \frac c2,\quad B_-+2\eta_-A_- = B_+-2\eta_+A_+ =0,
\end{equation}
so that we even have an equality in \eqref{v''}.
Plugging \eqref{condv''} into \eqref{condv0}, we find
\begin{equation}\label{condeta}
\frac c6 \eta_+^3 + \frac{2C}{3}\eta_-^3 =\beta,\quad 4C\eta_-^2 = c\eta_+^2,
\end{equation}
which leads us to choose
\begin{equation}\label{eta}
\eta_\pm = \alpha_\pm\beta^{1/3},
\end{equation}
where $\alpha_\pm>0$ are the solutions of
\begin{equation*}
4C\alpha_-^2 = c\alpha_+^2,\quad \frac c6 \alpha_+^3 + \frac{2C}{3}\alpha_-^3=1.
\end{equation*}
With $A_\pm$, $B_\pm$ and $\eta_\pm$ chosen as in \eqref{condv''}-\eqref{eta}, the function $v$ is of class $W^{2,\infty}$ and satisfies \eqref{v''}. Moreover, it is straightforward to check that
\begin{equation*}
|v'|+|v''|\leq C\beta^{1/3} \quad\text{ in }(-\eta_-,\eta_+),
\end{equation*}
so that \eqref{boundsv} is satisfied, which concludes the construction of $V_1$ satisfying \eqref{eqVj}. On the other hand $V_1$ obviously satisfies \eqref{boundsVj} since
\begin{equation*}
\lbrace \abs{V_1}<\beta/2\rbrace = \lbrace -\eta_- < z<\eta_+ \rbrace.
\end{equation*}

We omit the construction of $V_2$, which is completely similar to the one just performed.
\end{proof}

\section{Intermediate regimes}\label{s_interm}

As discussed in the introduction (Section~\ref{s_intro}), in the present section we want to understand the transitions occurring as $\beta$ decreases from $\beta_c$ to 0, when the set $\lbrace H=0\rbrace$ has more than one connected component.

In Section~\ref{ss_symcase} we study in detail a special case with rotational symmetry along a vertical axis, to provide some insight into the transition from the vortexless state to the zero solution. The reason to restrict to this setting is that it encapsules, what we believe are, the most interesting changes in the superconducting set that can occur. 

On the one hand, once we drop the assumption of rotational symmetry, changes in $H$ inside the sample could lead to arbitrarily intricate solutions to the obstacle problem for different values of $\beta,$ so a general theorem is not available. On the other hand the symmetries we consider highlight many model situations with remarkable properties.  One of these is the striking phenomenon that some parts of the free boundary may freeze: that is, remain constant with respect to $\beta$, for $\beta$ in some interval. In Section~\ref{ss_freez} we generalize this observation to the general, non-symmetric case. 

As mentioned earlier, a generalization of the other properties is precluded due to the wide variety of solutions one could construct, having the freedom to choose both $H$ and $\mathcal{M}.$ Nevertheless, we believe that under some more restrictive assumptions, in particular fixing the topology of the level sets of $H,$ one could extend the result on existence of the transitions observed in Proposition~\ref{prop_regimes1d}, however the role of the integral conditions on $I_{\pm}, J$ is not so easily transferable or even identifiable anymore.

\subsection{Detailed study of a symmetric case}\label{ss_symcase}

Here we consider a surface of revolution of the form
\begin{equation*}
\mathcal M =\left\lbrace (\rho(\phi)\cos \theta,\rho(\phi)\sin\theta,z(\phi))\colon \phi\in [0,\pi],\,\theta\in[0,2\pi] \right\rbrace,
\end{equation*}
where $\rho$ and $z$ are smooth functions linked by the relation 
\begin{equation*}
z(\phi)\tan\phi=\rho(\phi),
\end{equation*}
and satisfying $\rho(0)=\rho(\pi)=0$, $\rho>0$ in $(0,\pi)$,  $z'(0)=z'(\pi)=0$, and
\begin{equation*}
\gamma:=\sqrt{(\rho')^2+(z')^2}\geq c>0.
\end{equation*}
The volume form on such $\mathcal M$ is $d\mathcal H^2_{\mathcal M}=\rho\gamma d\theta d\phi$.

The induced magnetic potential $\mathbf A$ on $\mathcal M$ is also assumed to be symmetric, of the form
\begin{equation*}
\mathbf A = a(\phi) d\theta = \frac{a(\phi)}{u(\phi)} \hat e_\theta,
\end{equation*}
and we make the following assumptions on the functions $a$:
\begin{itemize}
\item[(a1):] $a(0)=a(\pi)=0$, and $a>0$ in $(0,\pi)$.
\item[(a2):] $a'>0$ in $(0,\phi_1)$ and $(\phi_2,\phi_3)$ and $a'<0$ in $(\phi_1,\phi_2)$ and $(\phi_3,\pi)$, for some $0<\phi_1<\phi_2<\phi_3<\pi$.
\end{itemize}
The function $a(\phi)$ hast two local maxima $a_1=a(\phi_1)$ and $a_3=a(\phi_3)$, and one local minimum $a_2=a(\phi_2)$. To simplify notations to come, we assume in addition that $a_1<a_3$. See Figure~\ref{figa1}.

\begin{figure}[ht]
\begin{center}
\includegraphics[width=6cm]{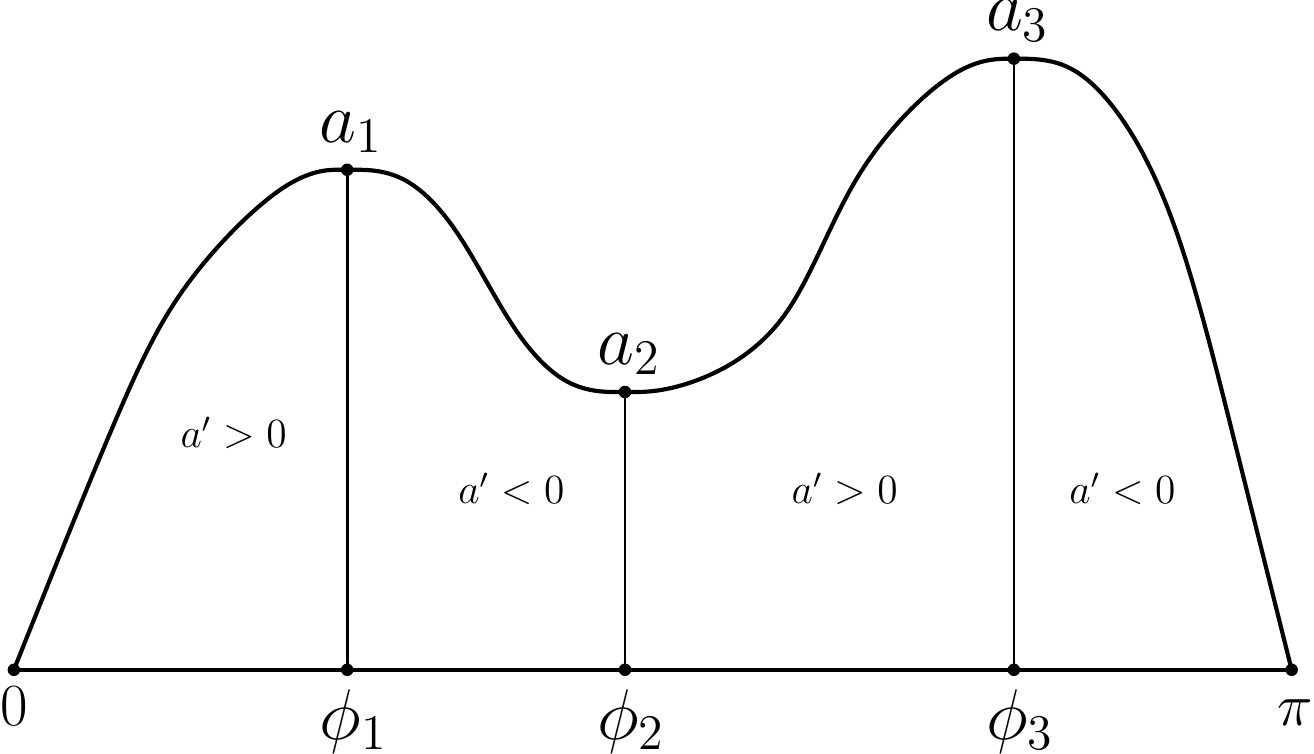}
\caption{The shape of $a(\phi)$.}
\label{figa1}
\end{center}
\end{figure}

\begin{remark}\label{rem_unifhex}
The case, presented in the introduction (Section~\ref{s_intro}), of a uniform external magnetic field $\mathbf H^e =\mathbf{e_z}$ corresponds to $a=u^2/2$.
\end{remark}

In that setting, the functions $H$ and $*F$ are also axially symmetric: they depend only on $\phi$ , and are given by
\begin{equation*}
H=\frac{a'}{\rho\gamma},\quad  (*F)'= a\frac{\gamma}{\rho}.
\end{equation*}

By uniqueness (up to a possible additive constant), the solution of the obstacle problem \eqref{obstacleH} is also rotationally symmetric: it holds $V=v(\phi)$. Since $V\in C^1(\mathcal M)$, the function $v$ should satisfy
\begin{equation*}
v\in C^1([0,\pi]),\quad v'(0)=v'(\pi)=0.
\end{equation*} 
Moreover, the free boundary problem \eqref{freebound} becomes
\begin{equation}\label{1dfreebound}
\left\lbrace
\begin{aligned}
\abs{v}&\leq \beta/2 &\text{in }[0,\pi],\\
\left(\rho\gamma^{-1} v'-a\right)'&=0 & \text{in }\lbrace |u|<\beta/2\rbrace, \\
a'&\geq 0 & \text{in }\lbrace v=-\beta/2\rbrace,\\
a'&\leq 0 &\text{in }\lbrace v=\beta/2\rbrace.
\end{aligned}
\right.
\end{equation}

We investigate, for $\beta<\beta_c$, the changes in the shape of the superconducting set $SC_\beta=\lbrace \abs{v}<\beta/2\rbrace$.
 The critical values at which that shape changes depend on the values of integrals $\int a\,\gamma\rho^{-1} d\phi$ on some intervals related to the level sets of $a(\phi)$.
That is why we start by fixing some notations concerning the level sets of $a(\phi)$. There are three different cases, depicted in Figure~\ref{figlevela}:
\begin{itemize}
\item For $\alpha\in (0,a_2)$, $\lbrace a=\alpha\rbrace =\lbrace \phi_- < \phi_+\rbrace$.
\item For $\alpha\in (a_2,a_1)$, $\lbrace a=\alpha\rbrace =\lbrace \phi_- < \psi_+ < \psi_- < \phi_+\rbrace$.
\item For $\alpha\in (a_1,a_3)$, $\lbrace a=\alpha\rbrace = \lbrace \psi_- < \phi_+\rbrace$. 
\end{itemize}
The functions $\phi_\pm(\alpha)$, $\psi_\pm(\alpha)$ are continuous on their intervals of definition.

For $\alpha\in (a_2,a_1)$, we define
\begin{equation}\label{defIJ}
\begin{gathered}
I_-(\alpha)=\int_{\phi_-}^{\psi_+} (a-\alpha)\frac\gamma\rho\, d\phi,\quad I_+(\alpha)=\int_{\psi_-}^{\phi_+} (a-\alpha)\frac\gamma\rho\, d\phi,\\
J(\alpha)=-\int_{\psi_+}^{\psi_-} (a-\alpha)\frac\gamma\rho\, d\phi.
\end{gathered}
\end{equation}

Those integrals corresponds to ``weighted'' areas of the regions depicted in Figure~\ref{figIJ}, with respect to the measure $\gamma\rho^{-1} d\phi$. Note that both the integrands and the intervals of integration depend on $\alpha$.

We identify a critical value of $\alpha$ with respect to these integrals.
\begin{lemma}\label{critalpha}
There exists $\alpha_*\in (a_2,a_1)$ such that:
\begin{itemize}
\item for $a_2<\alpha<\alpha_*$, $J<\min(I_\pm)$.

\item for $\alpha_*<\alpha<a_1$, $\min(I_\pm) < J$.
\end{itemize}
\end{lemma}

\begin{proof}
It follows from the obvious facts that $J$ is increasing, $I_\pm$ are decreasing, $J(a_2)=0$, $I_-(a_1)=0$, and the functions are continuous.
\end{proof}

\begin{figure}[ht]
\begin{center}
\begin{subfigure}{.4\textwidth}
\includegraphics[width=\textwidth]{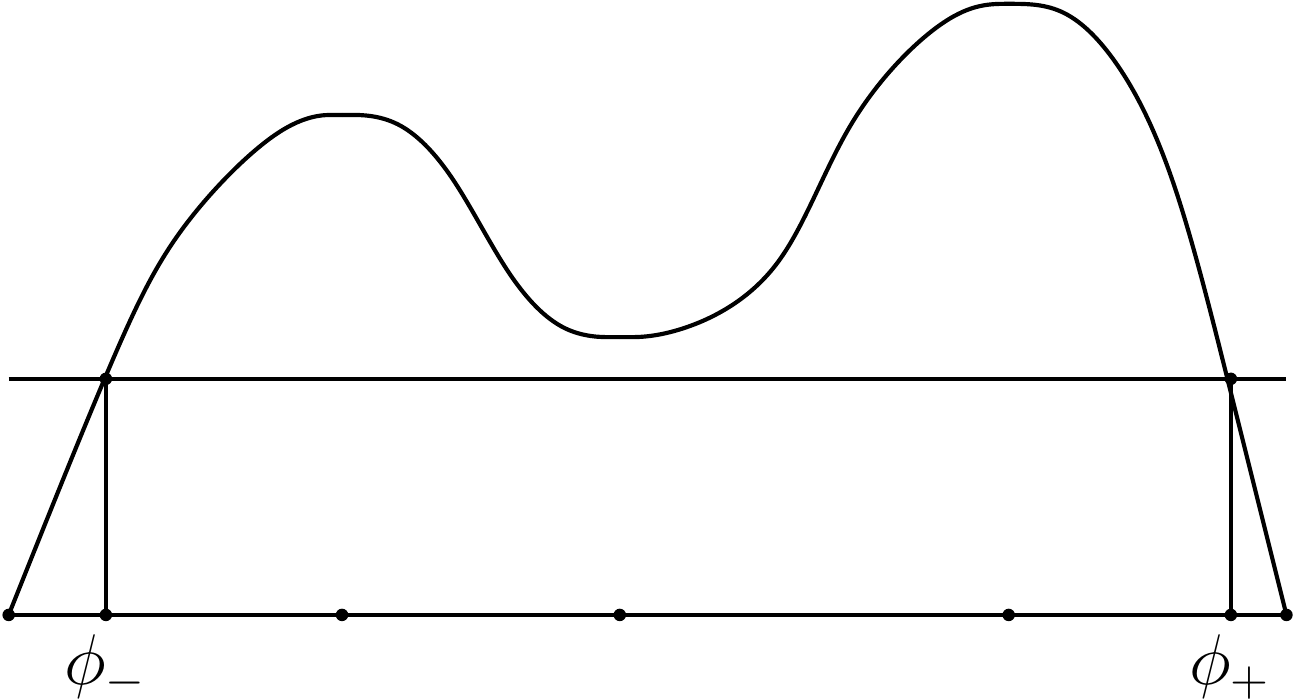}
\caption{$\alpha\in (0,a_2)$}
\label{figa2}
\end{subfigure}
\hspace{.05\textwidth}
\begin{subfigure}{.4\textwidth}
\includegraphics[width=\textwidth]{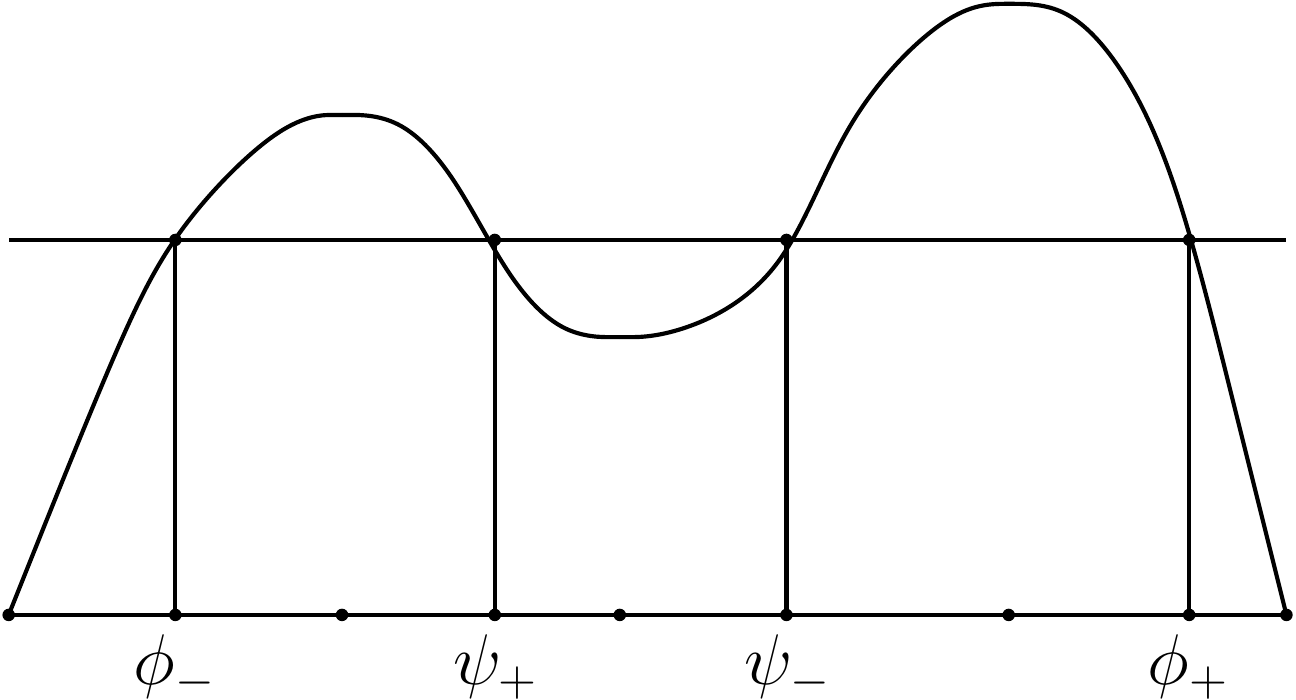}
\caption{$\alpha\in (a_2,a_1)$}
\label{figa3}
\end{subfigure}

\vspace{1em}
\begin{subfigure}{.4\textwidth}
\includegraphics[width=\textwidth]{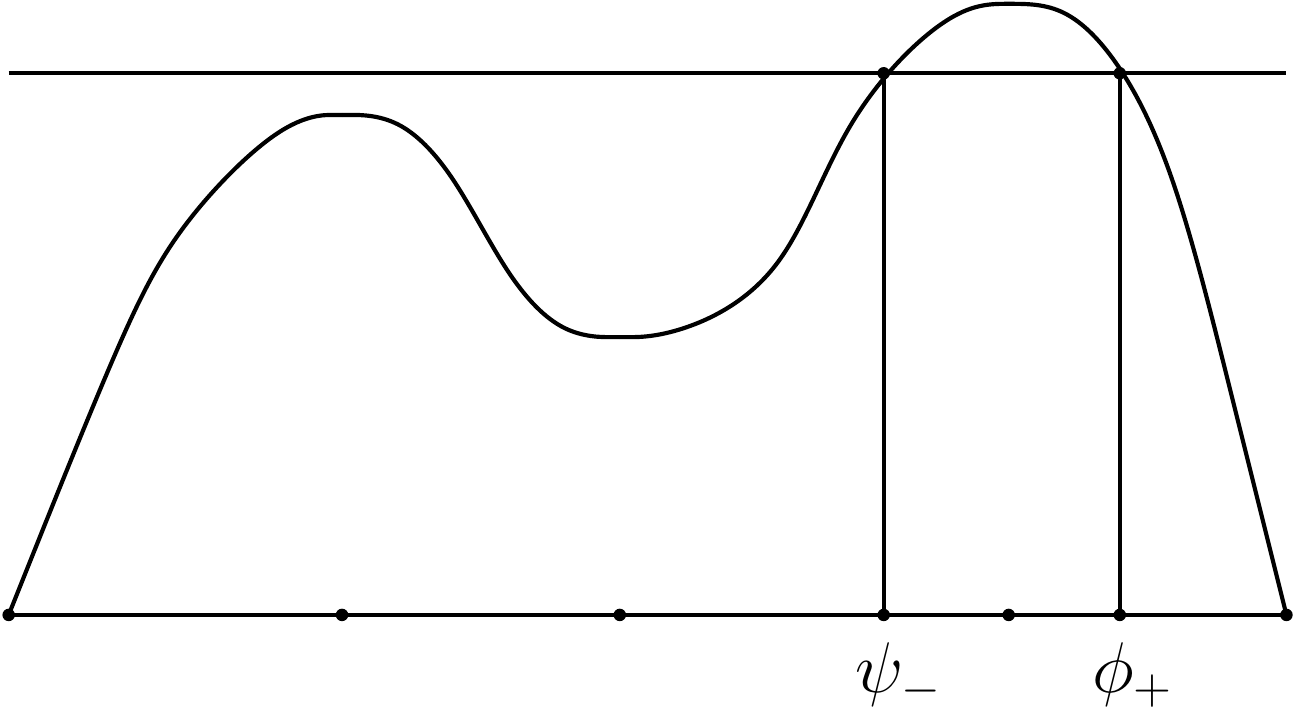}
\caption{$\alpha\in (a_1,a_3)$}
\label{figa4}
\end{subfigure}
\caption{Level sets $\lbrace a=\alpha\rbrace$}
\label{figlevela}
\end{center}
\end{figure}

\begin{figure}[ht]
\begin{center}
\includegraphics[width=6cm]{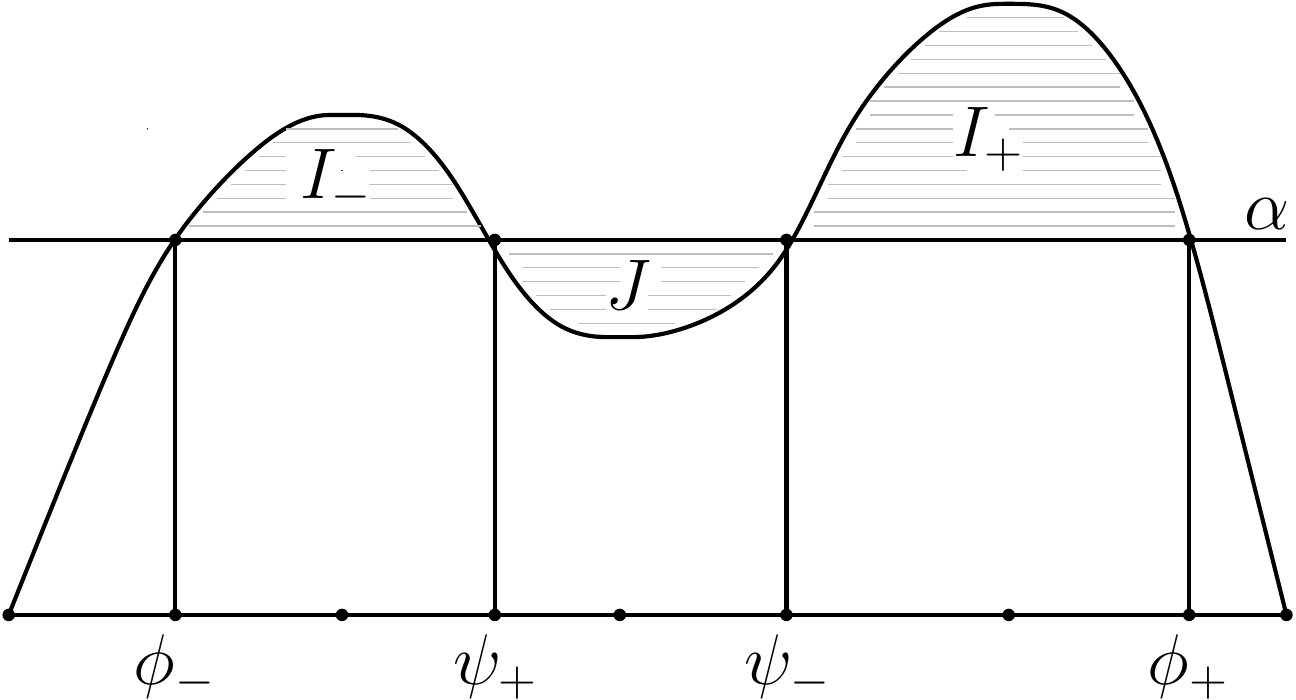}
\caption{The integrals $I_\pm$ and $J$.}
\label{figIJ}
\end{center}
\end{figure}

Now we may give the precise version of Proposition~\ref{loose1d}.
\begin{prop}\label{prop_regimes1d}
Let $\beta_c>\beta^*_1 \geq \beta^*_2 >0$ be defined by
\begin{equation*}
\beta^*_1:=\max(I_\pm(\alpha^*)),\quad \beta^*_2:=\min (I_\pm(\alpha^*)).
\end{equation*}
Then the conclusion of Propostion~\ref{loose1d} holds:
\begin{itemize}
\item For $\beta_c >\beta >\beta^*_1$,  $SC_\beta$ is an interval.
\item For $\beta^*_1>\beta>\beta^*_2$, $SC_\beta$ is the union of two disjoint intervals, one of them independent of $\beta$. 
\item For $\beta^*_2>\beta>0$, $SC_\beta$ is the union of three disjoint intervals.
\end{itemize}
\end{prop}

\begin{remark}
It may happen that $I_-(\alpha^*)=I_+(\alpha^*)$. In that case, $\beta^*_1=\beta^*_2$ and the second regime predicted by Proposition~\ref{prop_regimes1d} never happens.
\end{remark}

\begin{proof}[Proof of Proposition~\ref{prop_regimes1d}:]
By uniqueness (see Remark~\ref{rem_obstacle}), it suffices to exhibit, for each regime listed in Proposition~\ref{prop_regimes1d}, a solution of \eqref{1dfreebound} satisfying the desired properties.

\textbf{Case 1:} $\beta\in (\beta^*_1,\beta_c)$. The function
\begin{equation*}
I(\alpha):=\int_{\phi_-}^{\phi_+} (a-\alpha)\frac\gamma\rho\, d\phi,\quad\alpha\in (0,a_1),
\end{equation*}
is continuous, decreasing and satisfies $I(0)=\beta_c$ and $I(\alpha^*)=\beta^*_1$. Therefore there exists a unique $\alpha\in (0,\alpha^*)$ such that $I(\alpha)=\beta$. We define
\begin{equation*}
v(\phi)=
\begin{cases}
-\beta/2 & \text{ for }\phi\in (0,\phi_-),\\
-\beta/2 + \int_{\phi_-}^{\phi} (a-\alpha)\frac\gamma\rho\, d\tilde\phi & \text{ for }\phi\in (\phi_-,\phi_+),\\
\beta/2 & \text{ for }\phi\in (\phi_+,\pi).
\end{cases}
\end{equation*}
The shape of the function $v$ is sketched in Figure~\ref{figu1}.

\begin{figure}[ht]
\begin{center}
\includegraphics[width=8cm]{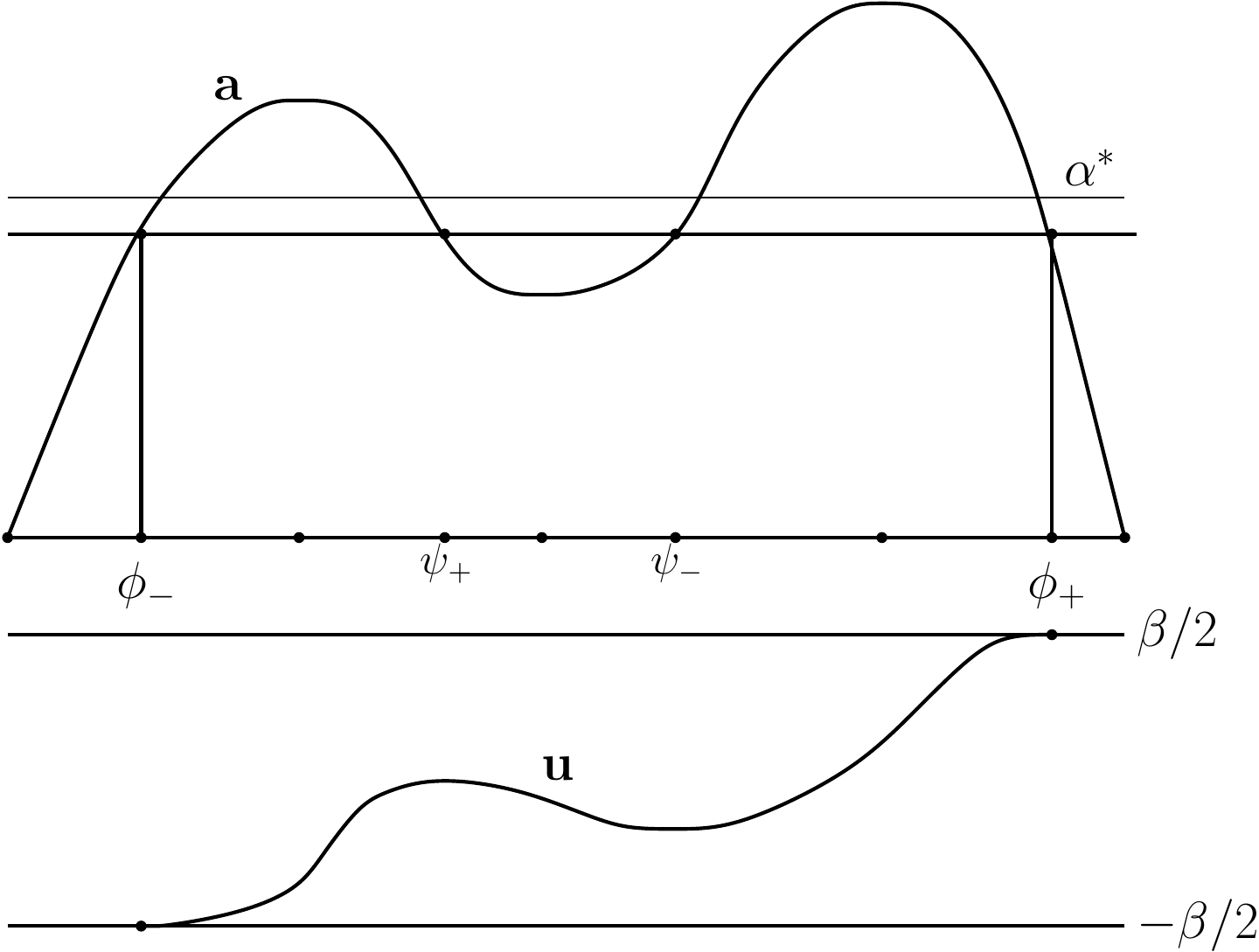}
\caption{The shape of $v$ for $\beta\in(\beta^*_1,\beta_c)$}
\label{figu1}
\end{center}
\end{figure}

The function $v$ is clearly continuous since $\beta$ has been chosen accordingly. Moreover, it holds
\begin{equation*}
v'(\phi_+)=(a(\phi_+)-\alpha)\frac\gamma\rho = v'(\phi_-)=0,
\end{equation*}
since by definition $a(\phi_+)=a(\phi_-)=\alpha$. Hence $v$ is in fact $C^1$ in $[0,\pi]$. Also by definition, $a'\geq 0$ in $(0,\phi_-)$ and $a'\leq 0$ in $(\phi_+,\pi)$. In addition, we clearly have $(\rho\gamma^{-1}v'-a)'=0$ in $(\phi_-,\phi_+)$. To prove that $v$ solves \eqref{1dfreebound}, it only remains to show that $|v|<\beta/2$ in $(\phi_+,\phi_-)$. We consider two different cases, depending on whether $\alpha\in (0,a_2]$ or $\alpha\in (a_2,\alpha^*)$.

If $\alpha\in (0,a_2)$, then (see Figure~\ref{figa2})
\begin{equation*}
v'=(a-\alpha)\frac\gamma\rho>0\quad\text{in }(\phi_-,\phi_+),
\end{equation*}
so that $v$ is increasing on $(\phi_-,\phi_+)$ and it clearly holds $|v|<\beta/2$. For $\alpha=a_2$ the derivative $v'$ only vanishes at one point and the same conclusion is valid.

If, on the other hand $\alpha\in (a_2,\alpha^*)$, then (see Figure~\ref{figu1})
\begin{equation*}
v'=(a-\alpha)\frac\gamma\rho\begin{cases}
>0 &\text{ in }(\phi_-,\psi_+),\\
<0 & \text{ in }(\psi_+,\psi_-),\\
>0 & \text{ in }(\psi_-,\phi_+).
\end{cases}
\end{equation*}
Therefore it suffices to check that $v(\psi_+)<\beta/2$ and $v(\psi_-)>-\beta/2$. We have, since $I(\alpha)=\beta$ and by definition of $I_\pm$ and $J$ (see Figure~\ref{figIJ}),
\begin{equation*}
\begin{split}
v(\psi_+)-\beta/2 & = I_-(\alpha)-\beta = I_-(\alpha)-I(\alpha)=J(\alpha)-I_+(\alpha),\\
v(\psi_-)+\beta/2 & = I_-(\alpha)-J(\alpha).
\end{split}\end{equation*}
Since $\alpha<\alpha^*$ we find indeed (by definition of $\alpha^*$) that $v(\psi_+)<\beta/2$ and $v(\psi_-)>-\beta/2$, and in that case also we conclude that $v$ solves the free boundary problem
\eqref{1dfreebound}.

\textbf{Case 2:} $\beta\in (\beta^*_2,\beta^*_1)$. We treat the case where $\min (I_\pm(\alpha^*))=I_-(\alpha^*)$. Thus $\beta^*_1=I_+(\alpha^*)$ and $\beta^*_2=I_-(\alpha^*)$. The other case can be dealt with similarly. 

The function $I_+(\alpha)$ is continuous and decreasing on $(a_2,a_3)$ and satisfies $I_+(\alpha^*)=\beta^*_1$ and $I_+(a_3)=0<\beta^*_2$ (see Figure~\ref{figIJ}). Therefore there exists $\alpha>\alpha^*$ such that $I_+(\beta)=\alpha$. We denote by $\psi_-$ and $\phi_+$ the two points of $\lbrace a=\alpha \rbrace \cap (\phi_2,\pi)$, and by $\phi_-^*<\psi_+^*<\psi_-^*$ the three  points of $\lbrace a=\alpha^*\rbrace \cap (0,\phi_3)$ (as in Figure~\ref{figu2} below). Note that, since $\alpha>\alpha^*$, $\psi_-^*<\psi_-$. Next we define
\begin{equation*}
v(\phi)=
\begin{cases}
-\beta/2 &\text{ for }\phi\in (0,\phi_-^*),\\
-\beta/2 +\int_{\phi_-^*}^{\phi}(a-\alpha^*)\frac\gamma\rho\, d\tilde\phi &\text{ for }\phi\in (\phi_-^*,\psi_-^*),\\
-\beta/2 & \text{ for }\phi\in (\psi_-^*,\psi_-),\\
-\beta/2 +\int_{\psi_-}^{\phi}(a-\alpha)\frac\gamma\rho\, d\tilde\phi & \text{ for }\phi\in (\psi_-,\phi_+),\\
\beta/2 & \text{ for }\phi\in(\phi_+,\pi).
\end{cases}
\end{equation*}
The shape of the function $v$ is sketched in Figure~\ref{figu2}.

\begin{figure}[ht]
\begin{center}
\includegraphics[width=8cm]{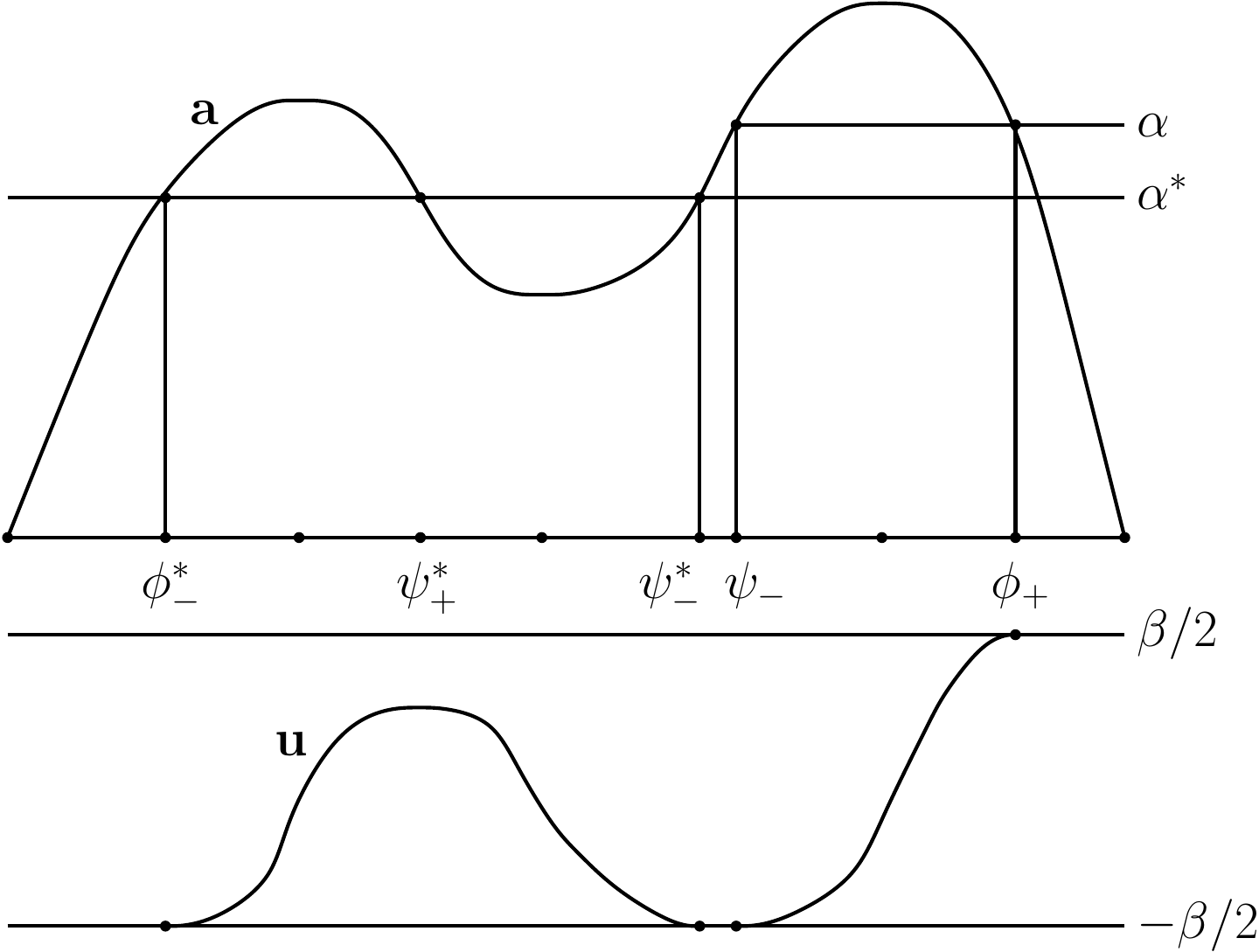}
\caption{The shape of $v$ for $\beta\in(\beta^*_2,\beta^*_1)$}
\label{figu2}
\end{center}
\end{figure}

Continuity of $v$ at $\psi_-^*$ is ensured by the fact that $I_-(\alpha^*)=J(\alpha^*)$. Continuity at $\phi_+$ by $I_+(\alpha)=\beta$. The function $v$ is $C^1$ because the facts that $a(\phi_-^*)=a(\psi_-^*)=\alpha^*$ and $a(\psi_-)=a(\phi_+)=\alpha$ guarantee that $v'(\phi_-^*)=v'(\psi_-^*)=v'(\psi_-)=v'(\phi_+)=0$. The sign of $a'$ is positive in $(0,\phi_-^*)$ and $(\psi_-^*,\psi_-)$ and negative in $(\phi_+,\pi)$.  In the two intervals $(\phi_-^*,\psi_-^*)$ and $(\psi_-,\phi_+)$, the equation $(\rho\gamma^{-1}v'-a)'=0$ is obviously satisfied, and it remains to check that $|v|<\beta/2$ in those intervals.

Since $v'=(a-\alpha)\gamma\rho^{-1} >0$ in $(\psi_-,\phi_+)$, it clearly holds $|v|<\beta/2$ in $(\psi_-,\phi_+)$.

In the interval $(\phi_-^*,\psi_-^*)$, the sign of $v'$ shows that $v$ attains its minimum at the boundary and its maximum at $\psi_+^*$, and it holds
\begin{equation*}
v(\psi_+^*)-\beta/2=-\beta + I_-(\alpha^*) =-\beta + \beta^*_2 <0.
\end{equation*}
We conclude that $v$ solves the free boundary problem \eqref{1dfreebound}. Moreover, the interval $(\phi_-^*,\psi_-^*)$ clearly does not depend on $\beta$.

\textbf{Case 3:} $\beta\in (0,\beta^*_2)$. 
Since $I_-$ is continuous and decreasing, $I_-(\alpha^*)>\beta^*_2$ and $I_-(a_1)=0$, there exists $\alpha_1>\alpha^*$ such that $I_-(\alpha_1)=\beta$. Similarly, there exist $\alpha_2<\alpha^*$ and $\alpha_3>\alpha^*$ such that $J(\alpha_2)=I_+(\alpha_3)=\beta$. We denote by
\begin{equation*}
0<\phi_-^1<\psi_+^1 <\psi_+^2 <\psi_-^2 <\psi_-^3 <\phi_+^3<\pi
\end{equation*}
the points such that (see Figure~\ref{figu3})
\begin{equation*}
\begin{gathered}
\lbrace a=\alpha_1\rbrace\cap (0,\phi_2)=\lbrace \phi_-^1,\psi_+^1\rbrace,\\
\lbrace a=\alpha_2\rbrace\cap (\phi_1,\phi_3) =\lbrace \psi_+^2,\psi_-^2\rbrace,\\
\lbrace a=\alpha_3\rbrace\cap (\phi_2,\pi)=\lbrace \psi_-^3,\phi_+^3\rbrace.
\end{gathered}\end{equation*}

\begin{figure}[ht]
\begin{center}
\includegraphics[width=8cm]{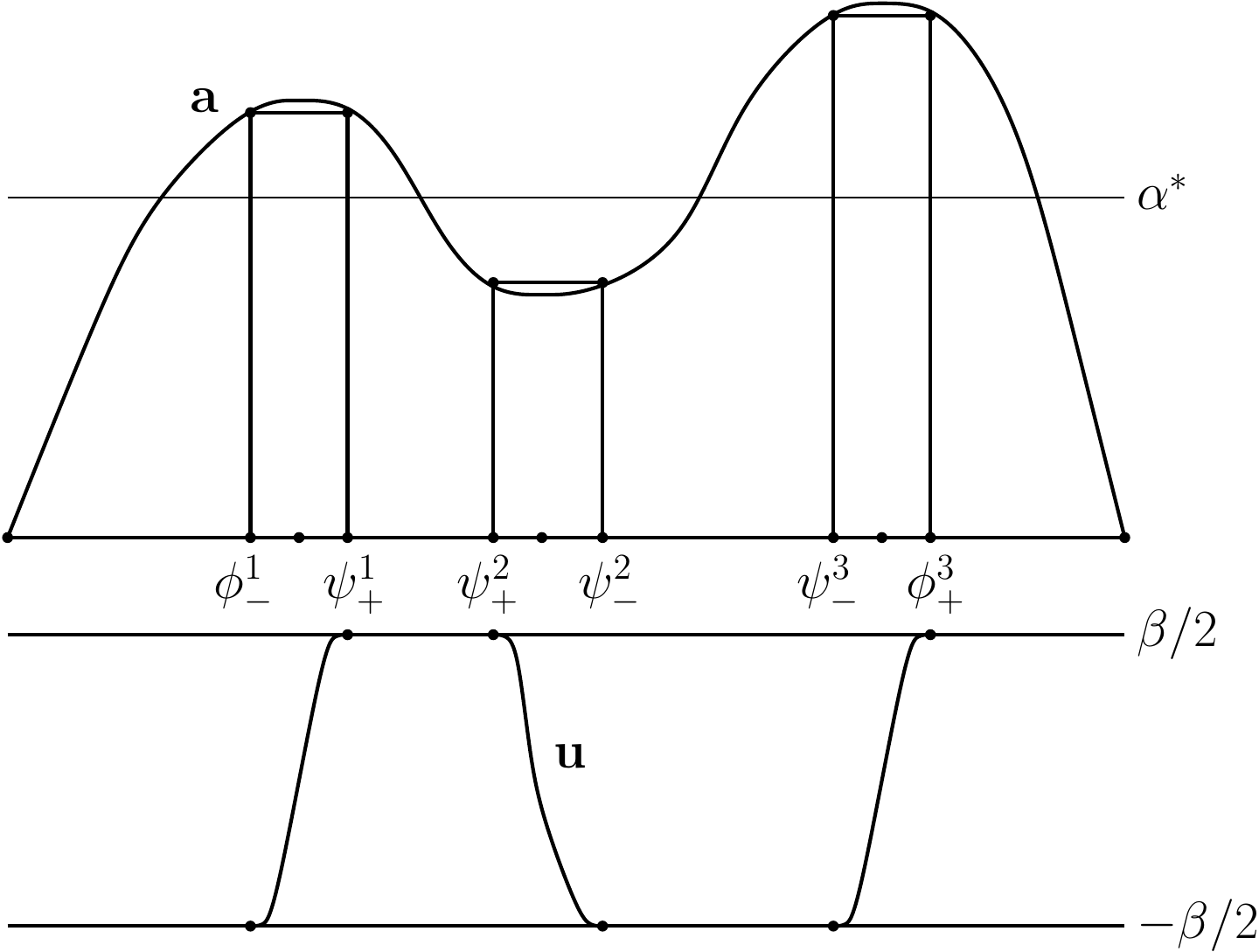}
\caption{The shape of $v$ for $\beta\in(0,\beta^*_2)$}
\label{figu3}
\end{center}
\end{figure}

Then we define
\begin{equation*}
v(\phi)=\begin{cases}
-\beta/2 &\text{ for }\phi\in (0,\phi_-^1)\text{ or }\phi\in (\psi_-^2,\psi_-^3)\\
-\beta/2 +\int_{\phi_-^1}^{\phi}(a-\alpha_1)\frac\gamma\rho\, d\tilde\phi &\text{ for }\phi\in (\phi_-^1,\psi_+^1),\\
\beta/2 & \text{ for }\phi\in (\psi_+^1,\psi_+^2)\text{ or }\phi\in (\phi_+^3,\pi),\\
\beta/2 +\int_{\psi_+^2}^{\phi}(a-\alpha_2)\frac\gamma\rho\, d\tilde\phi & \text{ for }\phi\in (\psi_+^2,\psi_-^2)\\
-\beta/2 +\int_{\psi_-^3}^{\phi}(a-\alpha_3)\frac\gamma\rho\, d\tilde\phi & \text{ for }\phi\in (\psi_-^3,\phi_+^3).
\end{cases}
\end{equation*}
The shape of the function $v$ is sketched in Figure~\ref{figu3}.

As above the $C^1$ regularity of $v$ follows from the definitions of $\alpha_1$, $\alpha_2$ and $\alpha_3$. The sign of $a'$ is positive in $(0,\phi_-^1)\cup (\psi_-^2,\psi_-^3)$ and negative in $(\psi_+^1,\psi_+^2)\cup (\phi_+^3,\pi) $. The equation $(\rho\gamma^{-1}v'-a)'=0$ is satisfied in the three intervals $(\phi_-^1,\psi_+^1)$,  $(\psi_+^2,\psi_-^2)$ and $(\psi_-^3,\phi_+^3)$. Moreover in those intervals, the function $v$ is monotone, hence $|v|<\beta/2$. Therefore $v$ solves the free boundary problem \eqref{1dfreebound}.
\end{proof}

\subsection{`Freezing' of the free boundary}\label{ss_freez}

\begin{prop}\label{prop_freez}
Assume that, for some $\beta_0\in (0,\beta_c)$, one connected component $\omega$ of the superconductivity set $SC_{\beta_0}$ is such that $V_{\beta_0}$ takes the same value on each connected component of $\partial\omega$. Then there exists $\delta>0$ such that
\begin{equation}\label{freez}
SC_\beta \cap\overline \omega  = SC_{\beta_0}\cap\overline\omega =\omega,
\end{equation}
for all $\beta\in (\beta_0-\delta,\beta_0]$.
\end{prop}

In Figure~\ref{figfreez} we show a situation corresponding to Proposition~\ref{prop_freez}, with $V=-\beta/2$ on every connected component of $\partial\omega$.

\begin{figure}[ht]
\begin{center}
\includegraphics[width=4.5cm]{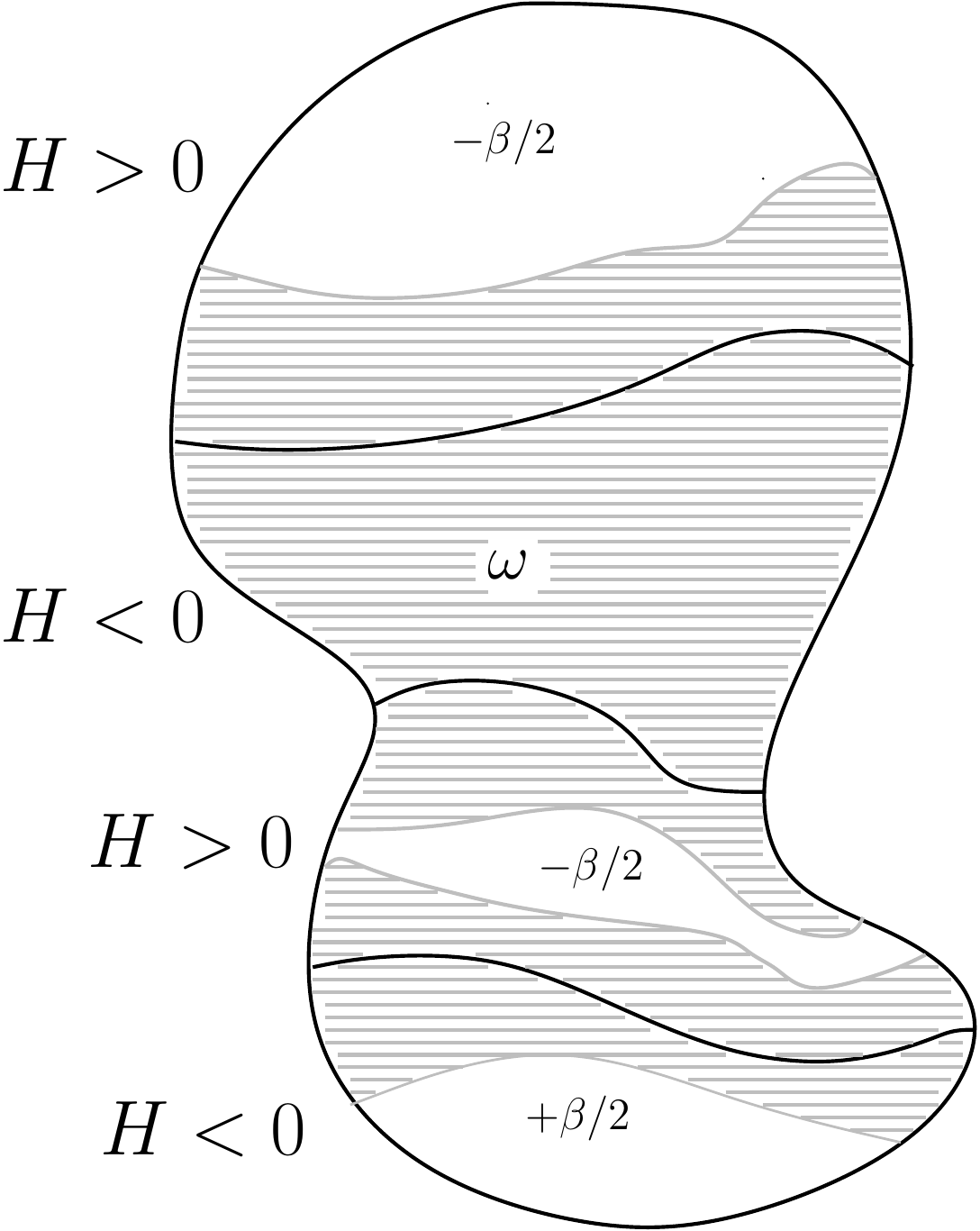}
\caption{An example of the situation of Proposition~\ref{prop_freez}}
\label{figfreez}
\end{center}
\end{figure}

\begin{remark}
The assumption on $\beta_0$ in Proposition~\ref{prop_freez} corresponds exactly to what happens in the symmetric case (Proposition~\ref{prop_regimes1d}) in the regime $\beta_1^*>\beta>\beta_2^*$, where $u(\phi_-^*)=u(\psi_-^*)=-\beta/2$ (Figure~\ref{figu2}).
\end{remark}

\begin{proof}[Proof of Proposition~\ref{prop_freez}:]
We present the proof in the case where $V=-\beta_0/2$ on every connected component of $\partial\omega$. The case $V=\beta_0/2$ on $\partial\omega$ can be dealt with similarly.

Since $V<\beta_0/2$ in $\omega$ and $V=-\beta_0/2$ on $\partial\omega$, it holds
\begin{equation*}
m:=\max_{\overline\omega} V <\beta_0/2,
\end{equation*}
and we define
\begin{equation*}
\delta:=  \frac 12 {\beta_0}-m >0.
\end{equation*}

Let $\beta\in (\beta_0-\delta,\beta_0]$, and define
\begin{equation}\label{tildeV0}
\widetilde V_0 := V_{\beta_0} + \frac{1}{2}(\beta_0-\beta).
\end{equation} 
The definitions of $m$ and $\delta$ ensure that it holds
\begin{equation}\label{ineqtildeV0}
-\beta/2 \leq \widetilde V_0 \leq \frac 12 {\beta_0} -\delta + \frac{1}{2}(\beta_0-\beta) < \beta/2\quad\text{in }\overline\omega.
\end{equation}
We claim that 
\begin{equation}\label{freezV}
V_\beta =\widetilde V_0 \quad\text{in }\overline\omega,
\end{equation}
which obviously implies \eqref{freez}. 

Note that the proof of Proposition~\ref{prop_monot} implies that it always holds
\begin{equation}\label{VleqtildeV}
V_\beta \leq \widetilde V_0 \quad\text{in }\mathcal M.
\end{equation}
Let $\omega_\beta = SC_\beta\cap \omega$, and 
\begin{equation}\label{Unonneg}
U:=\widetilde V_0 - V_\beta \geq 0.
\end{equation}
Note that $U\in C^{1,\alpha}(\overline\omega)$, and $U=0$ on $\partial\omega$ (since, by definition of $\omega$, $\widetilde V_0=0$ on $\partial\omega$). 

Let $\omega':=\omega\cap SC_\beta$. It holds
\begin{equation}\label{DeltaU}
\Delta U = H\one_{\omega\setminus\omega'}\quad\text{ in }\omega.
\end{equation}
From \eqref{ineqtildeV0} and \eqref{VleqtildeV} it follows that
\begin{equation*}
V_\beta <\beta/2 \quad\text{in }\overline \omega.
\end{equation*}  
Therefore, recalling the free boundary formulation \eqref{freebound}, we have $H\geq 0$ in $\omega\setminus\omega'$. In particular \eqref{DeltaU} implies that
\begin{equation*}
\Delta U\geq 0 \quad \text{ in }\omega.
\end{equation*}
Let $\varepsilon>0$ and consider
\begin{equation*}
\varphi:=\max(U-\varepsilon,0)\in H^1(\omega).
\end{equation*}
Recalling that $U\in C(\overline \omega)$ and $U=0$ on $\partial\omega$, we know that  $\varphi$ has compact support inside $\omega$. Thus we may integrate by part (without knowing anything about the regularity of $\partial\omega$) to obtain
\begin{equation*}
\int_\omega |\nabla \varphi|^2 =\int_\omega \nabla\varphi\cdot\nabla U = -\int_\omega \varphi \Delta U \leq 0,
\end{equation*}
and we deduce that $\varphi\equiv 0$ in $\omega$, which implies that $U\leq\varepsilon$ in $\omega$. Letting $\varepsilon\to 0$, we conclude that $U\leq 0$ in $\omega$, which, together with \eqref{Unonneg}, shows that \eqref{freezV} holds.
\end{proof}

\appendix
\section{The mean field approximation}\label{a_meanfield}
 Recall we assume $\mathcal M\subset \mathbb R^3$ is a closed compact surface homeomorphic to a sphere, $\mathbf A$ a $1$-form on $\mathcal M$ such that $\mathbf A=d^*F=*d*F$ for some smooth non constant $2$-form $F$, and $\GM$ the Ginzburg-Landau energy
\begin{equation*}
\GM (\psi)  = \int_{\mathcal M} \left|(\nabla -i h \mathbf A )\psi\right|^2 \, 
+ \frac{\kappa^2}{2}\int_{\mathcal M} (|\psi|^2-1)^2.
\end{equation*}
The parameter $\kappa>0$ is going to tend to $+\infty$, as is the strength of the applied field $h(\kappa)>0$.

If $\psi$ is a critical point of $\GM$, written locally as $\psi=\rho e^{i\varphi}$, then it holds
\begin{equation*}
d(\rho^2(d\varphi-hd^*F))=0.
\end{equation*}
 We deduce that there exists a function $V$ such that
\begin{equation*}
*dV=\rho^2 (hd^*F-d\varphi).
\end{equation*}
The function $V$ is uniquely defined up to an additive constant, which we may fix by imposing $\int_{\mathcal M} V  =0$. The function
\begin{equation*}
\mu = -\Delta (V-h*F) = -\Delta V +H
\end{equation*}
is the vortex density.

In this paper we appeal to a mean field approximation result proved by Sandier and Serfaty in \cite{SS1}.
In our case we also have to handle positive and negative measures $\mu_+, \mu_{-}$ with total zero mass $\mu_+ (\mathcal{M})-\mu_{-}(\mathcal{M})=0 .$ In this appendix we verify that under the additional constraints present in our context, we still have such a reduction.
For an intensity $h(\kappa)$ comparable to $\ln \kappa$, the mean field approximation consists in approximating the problem of minimizing $\GM$ by a limiting problem on the vorticity measure. The result also relates the 

\begin{prop}\label{propmfa}
Assume that $\beta:=\lim_{\kappa\to\infty}\frac{\ln\kappa}{h(\kappa)}\geq 0$ and $h(\kappa)=o(\kappa^2)$. Let $\psi_\kappa$ be a minimizer of $\GM$, and the corresponding $V_\kappa$ be defined as above. Then, up to a subsequence, as $\kappa\to\infty$,
\begin{equation*}
\frac{V_\kappa}{h(\kappa)} \text{ converges to }W_*,
\end{equation*}
weakly in $H^1$ (and strongly in $W^{1,q}$ for $q<2$) where $W_*$ minimizes the energy
\begin{equation*}
E_\beta(W)=\frac 12 \int_{\mathcal M} |\nabla W|^2\, d\mathcal H^2 + \frac{\beta}{2} \|-\Delta W+H\|_{TV},
\end{equation*}
over the set of all $W\in H^1(\mathcal M)$ such that $(-\Delta W+ H)$ is a Radon measure. Here $\|\mu\|_{TV}=|\mu|(\mathcal M)$ denotes the total variation norm of the Radon measure $\mu$.

Moreover, it holds $\GM(\psi_\kappa) = h(\kappa)^2 E_\beta(W_*) + o(h(\kappa)^2)$.
\end{prop}

We impose the normalization conditions $\int_{\mathcal M} W \,d\mathcal H^2 = \int_{\mathcal M} *F \, d\mathcal H^2 = 0$. Then with some slight abuse of notation $E_\beta(W)$ can be expressed in terms of $\mu=-\Delta W+H$,
as 
\begin{equation*}
E_\beta(W) =E_\beta (\mu) =\frac{\beta}{2} \|\mu\|_{TV} + \frac 12 \int_{\mathcal M} G(x,y) d(\mu-H)(x) d(\mu-H)(y),
\end{equation*}
 where $G(x,y)$ is the Green's function satisfying
\begin{equation*}
-\Delta_{\mathcal M} G(\cdot,y) = \delta_y - \frac{1}{\mathcal H^2(\mathcal M)}.
\end{equation*}
Here $\mu$ has to be a Radon measure of zero average since it comes from $\mu=-\Delta(W-*F)$, hence $\int_{\mathcal M} d\mu = 0$. 

Note that $E_\beta(\mu)$ may not be well-defined for every measure $\mu$, but at the end we will only need it to be well-defined for the particular $\mu_*$ associated to $W_*$ solving the obstacle problem~\eqref{obstacleH}, and this follows from the regularity theory for the obstacle problem (see Lemma~\ref{lem_freebound} and Appendix~\ref{a_prooffreebound}).

\subsection*{Sketch of the proof of the upper bound in Proposition \ref{propmfa}}

 The proof of the lower bound and compactness for minimizers follows directly from Theorems $7.1$ and $7.2$ in \cite{SS1}. We note that a by product of the analysis in \cite{SS1} is that $\frac{2\pi\sum_{i\in I}d_i\delta_{a_i}}{h}$ converges to $-\Delta\(\frac{V_\kappa}{h}-*F\)$ in the sense of measures and in $W^{1,p},$ for $p<2.$ 

The upper bound on the other hand is a little more delicate to adapt. Next we provide the details. The main tool to derive the upper bound in \cite{SS1} is a construction of measures $\mu_\kappa$ which approximate the measure $\mu_*$ minimizing $I_\beta$, and which are concentrated in balls of size $\kappa^{-1}$ each carrying a weight $2\pi$.
 Before stating the precise result, we introduce the functional $J=J_\beta$ 
\begin{equation}\label{J(mu)}
J(\mu):=\beta \|\mu\|_{TV} +  \int_{\mathcal M\times \mathcal M} G(x,y) d\mu(x) d\mu(y).
\end{equation}
 The following result then corresponds to Proposition~2.2 in \cite{SS1}.

\begin{prop}\label{approxmeas}
Let $\mu=\mu_+-\mu_-$ be the minimizer of $I_\beta$. Then, for $\kappa$ large enough, there exist points $a_{j,\pm}^\kappa$, $1\leq j \leq n_\pm(\kappa)$, such that
\begin{equation*}
n_\pm(\kappa)\sim \frac{h(\kappa)\mu_\pm(\mathcal M)}{2\pi},\qquad d(a_{j,\pm}^\kappa,a_{\ell,\pm}^\kappa)>4\kappa^{-1},
\end{equation*}
and, letting $\mu_\kappa^{j,\pm}$ be the uniform measure on  $\partial B(a_{j,\pm},\kappa^{-1})$ of mass $2\pi$, the measure
\begin{equation*}
\mu_\kappa:=\frac{1}{h(\kappa)}\sum_{j=1}^{n_+(\kappa)}\mu_\kappa^{j,+} - \frac{1}{h(\kappa)}\sum_{j=1}^{n_-(\kappa)}\mu_\kappa^{j,-}\quad\text{converges to }\mu,
\end{equation*}
in the sense of measures as $\kappa\to +\infty$. Moreover it holds $\int_{\mathcal M}d\mu_\kappa = 0$, and
\begin{equation}\label{approxenergy}
\limsup_{\kappa\to\infty} \int_{\mathcal M\times \mathcal M} G(x,y) d\mu_\kappa(x) d\mu_\kappa(y) \leq J(\mu),
\end{equation}
where $J=J_\beta$ is defined in \eqref{J(mu)}.
\end{prop}

Above, $d$ denotes geodesic distance and $\partial B$ denotes a geodesic circle accordingly.The zero average property $\int d\mu_\kappa =0$ is needed later to solve $-\Delta (V-*F)=\mu_\kappa$. It actually amounts to asking $n_+(\kappa)=n_-(\kappa)$.
The upper bound \eqref{approxenergy} is crucial to estimate the energy of the testing configuration constructed with help of the measures $\mu_\kappa$, and requires great care in the way the points $a^\kappa_j$ are distributed.

In \cite{SS1}, the authors consider non-negative measures defined on a domain in the plane, with no average constraint. Here we are dealing with measures on a surface having positive and negative parts, and, more importantly, satisfying the  zero average constraint.

Next we state a lemma that can be directly adapted from \cite[Proposition~2.2]{SS1}, which deals only with positive measures with support inside a coordinate neighborhood. Then we will explain how to use this lemma to obtain Proposition~\ref{approxmeas} above.

\begin{lem}{\cite[Proposition~2.2]{SS1}}\label{approxlem1}
Assume that $\mu$ is a non-negative Radon measure on $\mathcal M$, absolutely continuous with respect to the 2-dimensional measure on $\mathcal M$, and with support contained inside a coordinate neighborhood. Then, there exist points $a_j^\kappa$, $1\leq j\leq n(\kappa)$, with
\begin{equation*}
n(\kappa)\sim \frac{h(\kappa)\mu(\mathcal M)}{2\pi}\quad\text{and }d(a_j^\kappa,a_\ell^\kappa)> 4\kappa^{-1},
\end{equation*}
such that, with $\mu_\kappa^j$ the uniform measure of mass $2\pi$ on $\partial B(a_j^\kappa,\kappa^{-1})$, it holds
\begin{equation*}
\mu_\kappa =\frac{1}{h(\kappa)}\sum_{j=1}^{n(\kappa)}\mu_\kappa^j \quad\text{converges to }\mu,
\end{equation*}
and the upper bound \eqref{approxenergy} is satisfied.
\end{lem}

The proof of Lemma~\ref{approxlem1} is just a straightforward adaptation of \cite[Proposition~2.2]{SS1}, using the coordinate chart to transport their construction from the plane to our surface and general properties of the Green's function of the Laplacian  on a compact surface.

Next we explain how to deal with non-negative measures whose support does not lie inside a coordinate neighborhood.

\begin{lem}\label{approxlem2}
Assume that $\mu$ is a non-negative Radon measure on $\mathcal M$, absolutely continuous with respect to the $2$-dimensional measure on $\mathcal M$. Then the conclusion of Lemma~\ref{approxlem1} holds.
\end{lem}
\begin{proof}

\textit{Step 1:} We reduce to the case where the support of $\mu$ is a finite disjoint union of compact coordinate neighborhoods. Assume indeed that the conclusion of Lemma~\ref{approxlem2} holds for such measures. It is possible to construct a sequence $\mu_n$ of such measures, such that $0\leq \mu_n \leq \mu$ and $\mu_n$ converges to $\mu$. Indeed, just define $\mu_n=\one_{K_n} \mu$, where $K_n$ is a finite disjoint union of compact subsets of coordinate neighborhoods, and $\mu(\mathcal M \setminus K_n)\to 0$. Such a sequence $K_n$ exists because $\mathcal M$ is compact and the measure $\mu$ is inner regular.
For each $\mu_n$ we obtain a sequence $\mu_n^\kappa$ tending to $\mu_n$ with the good properties. After a diagonal process, we obtain a sequence $\mu_\kappa$ converging to $\mu$,
such that
\begin{equation*}
\limsup_{\kappa\to\infty} \int_{\mathcal M\times\mathcal M} G(x,y) d\mu_\kappa(x) d\mu_\kappa(y) \leq \liminf J(\mu_n).
\end{equation*}
It remains to show that the right-hand side is less than $J(\mu)$, which follows from $0\leq\mu_n\leq\mu$ and $G\geq 0$.

\textit{Step 2:} We prove Lemma~\ref{approxlem2} for $\mu$ that can be decomposed in the form
\begin{equation*}
\mu=\mu_1 + \cdots + \mu_N,
\end{equation*}
where the supports of the $\mu_j$ are inside disjoint compact coordinate neighborhoods, and each $\mu_j$ is non-negative and absolutely continuous with respect to $\mathcal H^2_{\mathcal M}$. Then one can apply Lemma~\ref{approxlem1} to each $\mu_j$ to obtain sequences 
$\mu_{j,\kappa}$ with the good properties. Then, defining $\mu_\kappa=\mu_{1,\kappa}+\cdots+\mu_{N,\kappa}$, one obtains
\begin{equation*}
\limsup \int G(x,y) d\mu_\kappa(x) d\mu_\kappa(y)  \leq \sum_{j} J(\mu_j) + \limsup \sum_{j\neq\ell} \int G(x,y) d\mu_{j,\kappa}(x) d\mu_{\ell,\kappa}(y).
\end{equation*}
Since the supports of distinct $\mu_j$ are disjoint and $G(x,y)$ is continuous outside the diagonal $\lbrace x=y\rbrace$, it holds
\begin{equation*}
\int G(x,y) d\mu_{j,\kappa}(x) d\mu_{\ell,\kappa}(y) \to \int G(x,y) d\mu_{j}(x) d\mu_{\ell}(y) \qquad\text{for }j\neq \ell,
\end{equation*}
and we conclude that
\begin{equation*}
\limsup \int G(x,y) d\mu_\kappa(x) d\mu_\kappa(y)  \leq \sum_{j} J(\mu_j) + \sum_{j\neq\ell} \int G(x,y) d\mu_{j}(x) d\mu_{\ell}(y)=J(\mu).
\end{equation*}
The proof is complete.
\end{proof}

Finally we deal with measures having positive and negative parts, and satisfying the zero average constraint.

\begin{lem}\label{approxlem3}
Let $\mu$ be a zero-average Radon measure on $\mathcal M$, absolutely continuous with respect to $\mathcal H^2_{\mathcal M}$. Then the conclusions of Proposition~\ref{approxmeas} hold. 
\end{lem}

\begin{proof}

\textit{Step 1:} It suffices to construct measures $\mu_{\kappa}$ satisfying all the conclusions of Proposition~\ref{approxmeas}, except for the zero average constraint. Assume indeed that we have such a sequence. Since $\mu$ satisfies the zero average constraint, it holds $\mu_+(\mathcal M)=\mu_-(\mathcal M)$ and we deduce that $n_+(\kappa)-n_-(\kappa) = o(h(\kappa))$. Up to considering a subsequence, we may assume that either $n_+(\kappa)\geq n_-(\kappa)$ for every $\kappa$ (or the opposite, but this is completely symmetric). We fix a compact  $K$ such that $\mu_+(K)>0$ and $K$ is disjoint from the support of $\mu_-$. Since $\mu_\kappa^+(K)$ converges to $\mu_+(K)$, the number of points $a_{j,+}^\kappa$ that are contained in $K$ for large $\kappa$ is larger than $c\cdot h(\kappa)$ for $c>0$. In particular it is larger that $n_+-n_-$, and we may define a measure $\tilde\mu_\kappa^+$ obtained from $\mu_\kappa^+$ by removing $(n_+-n_-)$ points $a_{j,+}^\kappa$ that lie inside $K$. The measure $\tilde\mu_\kappa = \tilde\mu_\kappa^+ -\mu_\kappa^-$ now satisfies the zero average condition, and since $n_+-n_-=o(h)$ the convergence $\tilde\mu_\kappa\to\mu$ still holds. It remains to prove that the upper bound 
\eqref{approxenergy} is satisfied also by $\tilde\mu_\kappa$. Since $G\geq 0$ and $0\leq \tilde\mu_\kappa^+ \leq \mu_\kappa^+$, it holds
\begin{equation*}
\begin{split}
\int G(x,y) d\tilde\mu_\kappa(x) d\tilde\mu_\kappa(y) & \leq \int G(x,y) d\mu_\kappa(x) d\mu_\kappa(y)\\
& \quad +2\int G(x,y) d(\mu_\kappa^+ -\tilde\mu_\kappa^+)(x) d\mu_\kappa^-(y).
\end{split}\end{equation*}
The last term converges to zero since $G$ is continuous outside the diagonal and $\mu_\kappa^+ -\tilde\mu_\kappa^+$ converges to zero and has support inside $K$ which is disjoint from the support of $\mu_-$. Hence we conclude that \eqref{approxenergy} holds.

\textit{Step 2:} As in Step 1 of Lemma~\ref{approxlem2}, we reduce to the case of a measure $\mu$ such that $\mu_+$ and $\mu_-$ have disjoint compact supports. Assume indeed that Lemma~\ref{approxlem3} holds for such measures, and consider, by truncating, monotone approximations $\mu_n^\pm$ of $\mu_\pm$, with disjoints compact supports and such that $0\leq \mu_n^\pm \leq \mu_\pm$. For each $n$ there exist measures $\mu_\kappa^n$ with the good properties, converging to $\mu_n:=\mu_n^+-\mu_n^-$. After a diagonal process, one obtains a sequence $\mu_\kappa$ such that
\begin{equation*}
\limsup_{\kappa\to\infty} \int_{\mathcal M\times\mathcal M} G(x,y) d\mu_\kappa(x) d\mu_\kappa(y) \leq \liminf J(\mu_n).
\end{equation*}
Since $G\geq 0$, by monotone convergence (or dominated convergence) terms of the form $\int G d\mu_n^\pm d\mu_n^\pm$ converge to $\int G d\mu^\pm d\mu^\pm$, so that
\begin{equation*}
 \int G(x,y) d\mu_n(x)d\mu_n(y) \longrightarrow \int G(x,y) d\mu(x)d\mu(y),
\end{equation*}
and we also have $\norm{\mu_n}\to\norm{\mu}$, so that $J(\mu_n)\to J(\mu)$ and we conclude that \eqref{approxenergy} holds.

\textit{Step 3:} We assume now that $\mu_+$ and $\mu_-$ have disjoint compact supports. Applying Lemma~\ref{approxlem2} to each of these non-negative measures, we can proceed exactly as in Step 2 of Lemma~\ref{approxlem2} to obtain the conclusion.
\end{proof}

With Lemma~\ref{approxlem3} at hand, the proof of Proposition~\ref{approxmeas} simply follows from the regularity theory for the obstacle problem (see Lemma~\ref{lem_freebound} and Appendix~\ref{a_prooffreebound}), which ensures in particular that the minimizing measure $\mu_*$ is absolutely continuous with respect to $\mathcal H^2_{\mathcal M}$.

Then the upper bound is obtained by constructing test configurations with vortices at the $a_{j,\pm}^\kappa$ as in the proof of \cite[Proposition 2.1]{SS1}. Those test configurations are obtained by solving $-\Delta (V_\kappa -h*F)=h\mu_\kappa$ and constructing the corresponding $\psi_\kappa$ which has modulus $1$ outside the balls $ B(a_{j,\pm}^\kappa,2\kappa^{-1})$'s, and phase given by $d\varphi_\kappa = hd^*F-*dV_\kappa$. \qed

\section{Proof of Lemma~\ref{lem_freebound}}\label{a_prooffreebound}

\begin{proof}[Proof of Lemma~\ref{lem_freebound}:]
Assume first that $V$ solves \eqref{varineq}. Then $V\in W^{2,p}(\mathcal M)$ for any $1<p<\infty$. This can be proven exactly as in \cite[Section~1.3]{petrosyan}, by remarking first that $V$ minimizes the functional
\begin{equation*}
\widetilde { \mathcal F}(V) =\int_{\mathcal M}\! \left(\abs{\nabla V}^2 + 2HV\one_{|V|\leq \beta/2}\right),
\end{equation*}
over the whole space $H^1(\mathcal M)$,
then approximating $\widetilde{\mathcal F}$ by a regularized version $\widetilde{\mathcal F}_\varepsilon$ of it, where the function $\Phi(t)=t\one_{|t|<\beta/2}$ is approximated by smooth functions $\Phi_\varepsilon$. Applying the Calderon-Zygmund estimates to the minimizers $V_\varepsilon$ of the regularized functionals, one obtains a uniform bound $\norm{V_\varepsilon}_{W^{2,p}}\leq C$ which allows to conclude, letting $\varepsilon\to 0$, that $V\in W^{2,p}$.

 Let $\mu:=-\Delta V + H$. For any smooth $\varphi\geq 0$ with support inside $\lbrace V>-\beta/2\rbrace$, we may apply \eqref{varineq} with $W=V-\varepsilon\varphi$ for small $\varepsilon$. We deduce that $\mu\leq 0$ in $\lbrace V>-\beta/2\rbrace$. Similarly, we show that $\mu\geq 0$ in $\lbrace V<\beta/2\rbrace$. It follows that \eqref{freebound} holds.

For the converse, assume that \eqref{freebound} holds. It implies in particular that
\begin{equation*}
\Delta V = H\one_{|V|<\beta/2}.
\end{equation*} 
Let $W\in H^1(\mathcal M)$ with $|W|\leq \beta/2$, and define $\varphi=W-V$, so that 
\begin{equation*}
\varphi \geq 0 \text{ in }\lbrace V=-\beta/2\rbrace,\quad\text{ and }\varphi\leq 0 \text{ in }\lbrace V=\beta/2\rbrace.
\end{equation*}
Integrating by parts, we obtain
\begin{equation*}
\begin{split}
\int_{\mathcal M}\! \left( \nabla V\cdot\nabla\varphi + H\varphi\right) & = 
\int_{\mathcal M}\! H\varphi(1-\one_{|V|<\beta/2}) \\
& = \int_{\lbrace V=-\beta/2\rbrace}\!\!\! H\varphi + \int_{\lbrace V=\beta/2\rbrace}\!\!\! H\varphi \\
& \geq 0,
\end{split}
\end{equation*}
where the last inequality comes from the inequalities satisfied by $H$ and $\varphi$ in $\lbrace V=\pm\beta/2\rbrace$. This shows that \eqref{varineq} is satisfied.
\end{proof}

\bibliographystyle{plain}
\bibliography{ref}

\end{document}